\documentclass{amsart}[12pt]
\usepackage[table]{xcolor}
\usepackage{float}
\usepackage[bookmarksopen,bookmarksdepth=2,backref=page]{hyperref}
\hypersetup{
    colorlinks=true,
    linkcolor=blue,
    filecolor=magenta, 
    citecolor=cyan,  
    urlcolor=cyan,
    pdftitle={Fields of definition for  triangle groups},
    pdfpagemode=FullScreen,
    hyperindex=true,
    }
\usepackage{etex,mathrsfs } 
\usepackage{verbatim}
\usepackage{url}
\usepackage{amssymb,amsmath,amsfonts,amsthm,epsfig,amscd,stmaryrd}
\usepackage{stmaryrd}
\usepackage{diagbox}
\usepackage{booktabs}

\def\Kbar{\overline{K}}
\def\OL{\mathcal{O}}
\def\e{\mathbf{e}}

\def\R{\mathbf{R}} 
\def\Z{\mathbf{Z}}

\def\Q{\mathbf{Q}}
\def\GG{\mathbf{G}}
\def\Qbar{\overline{\Q}}
\def\Gal{\mathrm{Gal}}

\def\tareesidedbox#1{\setbox0=\hbox{$#1$}\dimen0=\wd0 \advance\dimen0 by3pt\rlap{\hbox{\vrule height9pt width.4pt depth2pt \kern-.4pt\vrule height9.4pt width\dimen0 depth-9pt\kern-.4pt \vrule height9pt width.4pt depth2pt}} \relax \hbox to\dimen0{\hss$#1$\hss}}
\def\ho#1{\tareesidedbox{#1}}

\newif\iffinalrun
\iffinalrun
\else
 \fi
\iffinalrun
  \newcommand{\need}[1]{}
  \newcommand{\mar}[1]{}
\else
  \newcommand{\need}[1]{{\tiny *** #1}}
\newcommand{\mar}[1]{\marginpar{\raggedright\tiny  #1 }}\fi
\renewcommand\mathbb{\mathbf}

\numberwithin{equation}{subsection}
\def\numequation{\addtocounter{subsubsection}{1}\begin{equation}}

\counterwithin{figure}{equation}

\makeatletter\let\c@equation\c@subsubsection\makeatother

\newtheorem{theorem}[subsubsection]{Theorem}

\newtheorem{lemma}[subsubsection]{Lemma}

\newtheorem{iprob}{Problem}

\newtheorem{ithm}[iprob]{Theorem}

\newtheorem{iconj}{Conjecture}

\theoremstyle{definition}
\newtheorem{example}[subsubsection]{Example}
\newtheorem{df}[subsubsection]{Definition}
\newtheorem{remark}[subsubsection]{Remark}

\def\SL{\mathrm{SL}}
\def\PSL{\mathrm{PSL}}
\def\v{\mathbf{v}}
\def\w{\mathbf{w}}

\ifdefined\NewCommandCopy
  \NewCommandCopy\latexbibitem\bibitem
\else
  \let\latexbibitem\bibitem
  \usepackage{xparse}
\fi

\renewcommand{\bibitem}[2][]{
  \def\mykey{#1}
  \def\cmpkeyA{Kro57}
  \def\cmpkeyB{Sch73}
  \def\cmpkeyC{Fri92}
    \def\cmpkeyD{Poi82}
 \ifx\mykey\cmpkeyA
    \latexbibitem[Kro1857]{#2} 
     \else\ifx\mykey\cmpkeyB
    \latexbibitem[Sch1873]{#2}
  \else\ifx\mykey\cmpkeyC
    \latexbibitem[Fri1892]{#2}
      \else\ifx\mykey\cmpkeyD
    \latexbibitem[Poi1882]{#2}
   \else
    \latexbibitem[#1]{#2}
  \fi\fi\fi\fi
  }

\thanks{This project was supported in part by NSF Grant DMS-2001097.}

\begin{document}

\title{Fields of definition for  triangle groups as Fuchsian groups}

 \author[Frank Calegari]{Frank Calegari}
 \email{fcale@uchicago.edu}
 \address{The University of Chicago,
5734 S University Ave,
Chicago, IL 60637, USA}

\author[Qiankang Chen]{Qiankang Chen}
 \email{jackqkchen@163.com} 

\begin{abstract}
The  compact hyperbolic triangle
group \(\Delta(p,q,r)\) admits a canonical
representation to \(\PSL_2(\R)\)
(unique, up to conjugation) whose image is discrete,
that is, a Fuchsian group.
The trace field of this representation is
\[K = \Q(\cos(\pi/p), \cos(\pi/q), \cos(\pi/r)).\]
 We prove
that there are exactly eleven such groups which are 
conjugate to subgroups of \(\PSL_2(K)\). 
These groups are precisely the
triangle groups which belong to the ``Hilbert Series''
as coined by McMullen~\cite{Curt,Bilu}. 
Moreover, we prove that there are no
additional compact
hyperbolic triangle groups which are conjugate to
subgroups of \(\PSL_2(L)\) for \emph{any} totally real field \(L\).
This answers a
question first raised by Waterman and Machlachlan in~\cite{WM}.
These questions were also recently
studied by McMullen~\cite{Curt,Bilu}, who raised five
 (interrelated)
 conjectures concerning the Hilbert Series;
we prove all of these conjectures.
\end{abstract}

\maketitle

{\footnotesize
\setcounter{tocdepth}{1}
\tableofcontents
}

\section{Introduction}

\subsection{The main results} 

The study of the triangle groups 
\[\Delta = \Delta(p,q,r) = \langle  x,y | x^p, y^q, (xy)^r \rangle\]
dates back to the 19th century, beginning with the work of 
Schwarz~\cite{Schwarz} and Poincar\'{e}~\cite{MR1554584}
(with respect to complex uniformization) and later by
Fricke~\cite{Fricke} from a more arithmetic perspective. 
If the parameters \((p,q,r)\) satisfy the inequality
\begin{equation}
\label{hyperbolic}
\frac{1}{p} + \frac{1}{q} + \frac{1}{r} < 1,
\end{equation}
Then \(\Delta\) is isomorphic to a cocompact subgroup of \(\PSL_2(\R)\).
This representation corresponds to a tessellation of the upper half
plane \(\mathbf{H}\) by hyperbolic triangles with angles \(\pi/p\), \(\pi/q\), and \(\pi/r\).
It is well known (see~\cite[Prop~1]{Takeuchi})
that the embedding~\(\Delta \hookrightarrow \PSL_2(\R)\) 
is unique up to conjugation and (as a consequence)  is isomorphic to a subgroup of~\(\PSL_2(L)\)
for some number field \(L \hookrightarrow \R\); such a map
is given explicitly in~\cite[Equation (2.7)]{ClarkVoight}.
The field~\(L\) necessarily contains
the  trace field~\cite[\S4.9]{Reid}, \cite{Neumann}:
\begin{equation} \label{tracefield}
K = \Q(\cos(2 \pi/p), \cos(2 \pi/q), \cos(2 \pi/r)).
\end{equation}
If \(\gamma \in \PSL_2(\R)\), then the trace of any lift
to \(\SL_2(\R)\) is well-defined up to sign, and \(K\)
is the field generated by generated by these traces.
In~\cite{Fricke}, Fricke studied the \((2,3,7)\) triangle group in detail:
\begin{center}
\begin{figure}[H]
  \includegraphics[width=3.5cm]{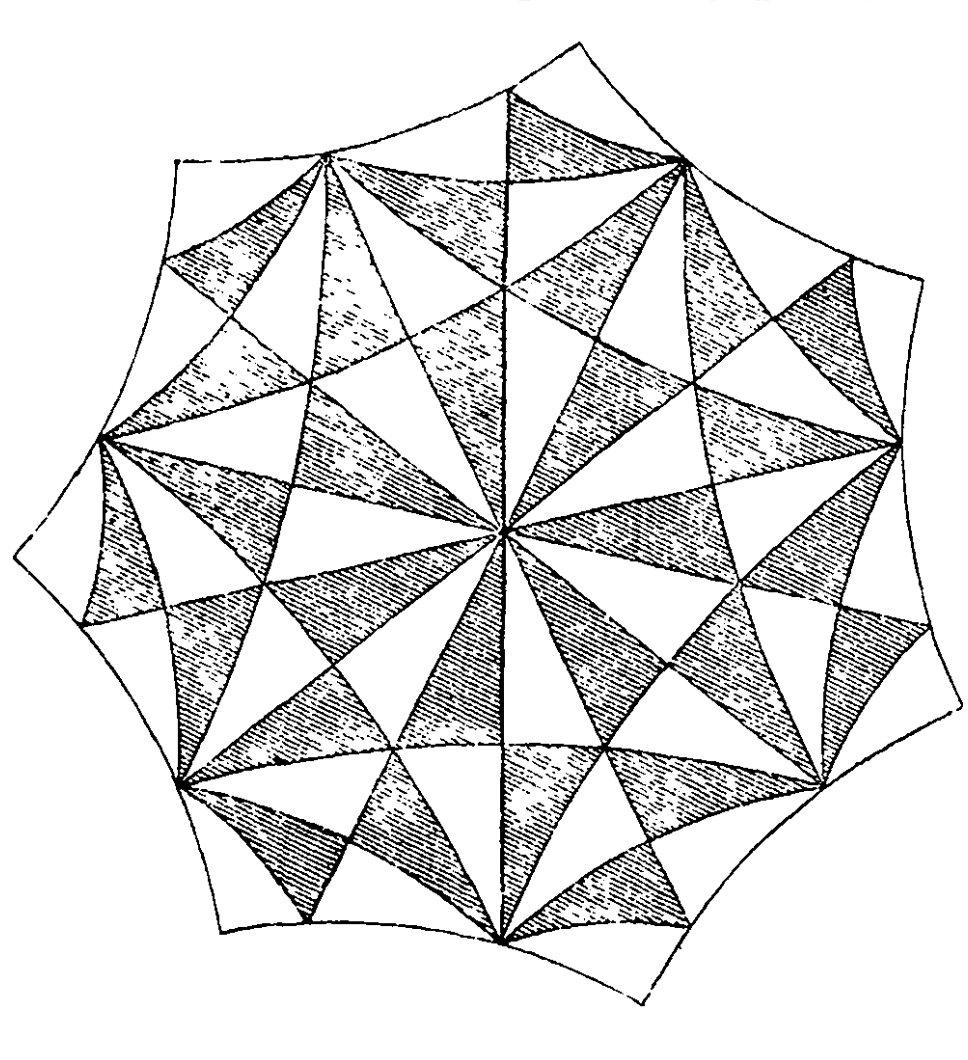}
  \caption{A (partial) tiling of hyperbolic space by \((2,3,7)\) triangles,
  taken from \cite[Fig~2]{Fricke}.}
  \label{fig:Fricke}
  \end{figure}
\end{center}
The field~\eqref{tracefield} specializes in this case to
\(K = \Q(\cos(2 \pi/7))\).
Fricke writes down an explicit representation of \(\Delta\) --- 
not over \(K\) but rather the
(non-Galois) quadratic extension \(L = K(j)\) of signature \((2,2)\)
 where \(j\) is
as follows  (see~\cite[Eq~(1)]{Fricke}):
\begin{equation} \label{frickequadratic}
j = \sqrt{e^{\tfrac{2 i \pi}{7}}
+ e^{-\tfrac{2 i \pi}{7}} - 1}
\end{equation}
This construction reflects the following: the group \(\Delta = \Delta(2,3,7)\)
is an arithmetic group corresponding to the quaternion
algebra \(B/K\) ramified at precisely two of the three real places (and no
finite places).
In particular, there is  obstruction to the existence
of an embedding \(\Delta \rightarrow \PSL_2(K)\) measured by \(B\),
and a representation to \(L/K\) exists if and only if \(L\) splits \(B\),
as occurs for Fricke's quadratic extension \(L = K(j)\) given 
in~\eqref{frickequadratic}.
The situation for  all compact arithmetic triangle groups is very
similar --- they correspond to quaternion algebras \(B\) 
over 
ramified
at all but one real place, and hence they never split over a totally
real field unless the corresponding base field
 is \(\Q\).\footnote{The base field may be smaller
 than \(K\) and coincides with the
 invariant trace field. In particular, in the notation
 introduced later, the corresponding
 quaternion algebra and field is \(B^{(2)}/K^{(2)}\).}
 On the other hand, 
 a well-known theorem of~\cite[Thm~3]{Takeuchi} implies
 that there only exist \(76\) cocompact
 hyperbolic triangle groups which are arithmetic.
This leads to the following natural question, which we answer
in this paper:
\begin{iprob} \label{WM} When is \(\Delta(p,q,r)\) isomorphic to a subgroup
of \(\PSL_2(K)\)?
\end{iprob}
This question was first considered by Waterman and Machlachlan in~\cite{WM},
where it was proved that Problem~\ref{WM}  has a negative answer
for all but finitely many 
compact triangle groups.\footnote{For the non-compact triangle groups
the situation is much simpler; there always exists a representation
over the corresponding trace field, see~\cite[Thm~1]{WM}.
This is because the existence
of parabolic elements in~\(\Delta\) forces the corresponding quaternion
algebra to split, cf.~\cite{Neumann}.}
 In fact, they prove the stronger statement
that only finitely many triangle groups are subgroups of \(\PSL_2(L)\)
for \emph{any} totally real finite extension \(L/K\).  Note that by construction,
they \emph{are} subgroups of \(\PSL_2(L)\) for some \(L/K\) with at least 
\emph{one} real embedding, but that is  much weaker than demanding
that \(L\) is totally real.

In this paper, we will give a complete answer to Question~\ref{WM};
there are no further examples beyond the ones found in~\cite{WM}.
This question has also recently been considered by McMullen~\cite{Bilu,Curt}
in relation to compact geodesic curves on Hilbert modular varieties. 
Before stating our main results, we introduce some notation
for the quaternion
algebras associated to \(\Delta\) and to its commensurability class.
The quaternion
algebras govern the fields \(L\) such that \(\Delta\) and its
finite order subgroups admit a representation to \(\PSL_2(L)\).
Let \(\Delta^{(2)} = \langle g^2 : g \in \Delta \rangle \) and let \(K^{(2)}\)
be the trace field of \(\Delta^{(2)}\). The field \(K^{(2)}\)
is the \emph{invariant} trace field of \(\Delta\) --- it is an invariant
of the commensurability class of \(\Delta\) (See~\cite{Reidoriginal,Reid}).
We have \([K:K^{(2)}] = 2^e\) where \(e = 2,1,0\) 
depending on whether three of the \((p,q,r)\) are even,
exactly two are even, or at most one is even respectively
(See~\cite{Takeuchi} and also~\cite{Curt},  where
 \(\Delta^{(2)}\) is denoted by \(\Delta_0\)).
The groups \(\Delta\) and  \(\Delta^{(2)}\) canonically admit
representations to \(\PSL(B)\) and \(\PSL(B^{(2)})\) respectively,
where \(B\)
and  \(B^{(2)}\) are the
associated quaternion algebras over \(K\) and \(K^{(2)}\).
The quaternion algebra of \(\Gamma\) is given explicitly
 by \(\Q[\widetilde{\Gamma}]\)
where \(\widetilde{\Gamma}\) is the pre-image of \(\Gamma\) in \(\SL_2(\R)\).
The representations of these groups are defined over
the respective trace fields if and only if the corresponding
 quaternion algebras split, and one obtains
 representations over some totally real extension \(L/K\) if and only
 if the quaternion algebras split at all real places.
 Thus question~\ref{WM} is equivalent to asking when \(B/K\)
 is split at all real places.
 (The ``opposite'' question  of understanding the triangle
 groups for which \(B/K\) is  \emph{ramified} at all but a fixed number
 of real
  places was considered in~\cite{VV}, following~\cite{Takeuchi}.)
In~\cite{Curt}, McMullen introduces the Hilbert Series consisting
of the following eleven hyperbolic triangle groups:

\begin{df}[McMullen, \cite{Curt}] The Hilbert Series consists
of the triangle groups \(\Delta\) with \((p,q,r)\) taken
from the following list:
\[\begin{aligned}
(2, 4, 6), (2, 6, 6), (3, 4, 4), (3, 6, 6), (2, 6, 10), (3, 10, 10), \\
(5, 6, 6), (6, 10, 15), (4, 6, 12), (6, 9, 18), \ \text{and} \  (14, 21, 42).
\end{aligned}
\]
\end{df}

(These groups were also identified in~\cite{WM}.)
McMullen then goes on to make the following five (related)
conjectures concerning the Hilbert Series for
hyperbolic triangle groups:

\begin{iconj}{\cite[Conjecture~1.3]{Curt},
\cite[Conjecture~1.15]{Bilu}} \label{conjA}
The quaternion algebra \(B_v\)  is split at all infinite places \(v\)
of \(K\) if and only if \(\Delta\)
 belongs to the Hilbert series.
 \end{iconj}

\begin{iconj}{\cite[Conjecture~1.4]{Curt}}  \label{conjB}
The quaternion algebra \(B\) is split 
if and only if \(\Delta\)
 belongs to the Hilbert series.
 \end{iconj}

\begin{iconj}{\cite[Conjecture~1.5]{Curt}}  \label{conjC}
The quaternion algebra \(B^{(2)}\) is split
if and only if \(\Delta\) is conjugate to \(\Delta(14,21,42)\).
 \end{iconj}
 
 \begin{iconj}{\cite[Conjecture~1.6]{Curt}}  \label{conjD}
 A cocompact triangle group \(\Delta\) has a model
 over a totally real field if and only if \(\Delta\) belongs
 to the Hilbert series.
 \end{iconj}
 
 \begin{iconj}{\cite[Conjecture~1.8]{Curt}} \label{conjE}
 A finite cover of~\(\mathbf{H}/\Delta\)  can be presented as
a Kobayashi geodesic curve on a Hilbert modular variety
if and only if \(\Delta\) belongs to the Hilbert series.
 \end{iconj}
 
 Conjecture~\ref{conjB} is equivalent to the claim
 that the answer to Problem~\ref{WM} is positive if and only if
 \(\Delta\) belongs to the Hilbert series.
 Conjecture~\ref{conjA} is equivalent to the claim
 that \(\Delta \subset \PSL_2(L)\) for some totally real field~\(L\)
 if and only if \(\Delta\) belongs to the Hilbert series.
The main theorem of this paper is as follows:
 
 \begin{ithm} \label{BBB} Conjectures~\ref{conjA}, \ref{conjB}, \ref{conjC}, \ref{conjD}, 
 and~\ref{conjE}
 are all true.
 \end{ithm}
 

\begin{ithm} \label{more} Let \(\Delta\) be a hyperbolic triangle group,
let \(K\) be the corresponding invariant trace field with
integer ring \(\OL_K\).
The following are equivalent:
\begin{enumerate}
\item \(\Delta\) belongs to the Hilbert series.
\item There exists a faithful representation
\(\Delta \rightarrow \PSL_2(\OL_K)\).
\item There exists a faithful representation
\(\Delta \rightarrow \PSL_2(K)\).
\item There exists a faithful representation
\(\Delta \rightarrow \PSL_2(L),\)
where \(L\) is  totally real.
\end{enumerate}
\end{ithm}

As established in~\cite{Curt}, Conjectures~\ref{conjB}, \ref{conjC}, \ref{conjD},
and~\ref{conjE} all follow
from Conjecture~\ref{conjA}.  
 Theorem~\ref{more} follows
from Theorem~\ref{BBB} together with~\cite[Thm 1.1]{Curt}.
As mentioned above, these results were previously known to hold \emph{up to finitely many
exceptions} by~\cite{WM}.  The argument in~\cite{WM} also gives,
in principle, an explicit (but huge) upper bound
on any counterexample. A new but softer proof of this finiteness result was
given by McMullen in~\cite{Bilu}. McMullen also verified~\cite[Thm~1.7]{Curt} that
the conjecture holds when the parameters \((p,q,r)\) are at most \(5000\).
The perspective of~\cite{Bilu} is to recognize the finiteness
as a  case of an equidistribution result for roots of unity.
Results of this kind are natural extensions of 
Lang's conjecture (proved
by Ihara, Serre, Tate in the 60's) concerning the intersection
of subvarieties of \(\GG^r_m\) with the set of torsion points.
In addition to these soft equidistribution results,
McMullen had to understand
the  geometry of a moduli space 
 of triangles identified
with \(\R^3/\Lambda\), where \(\Lambda \subset \Z^3\)
is the lattice consisting of triples \((a,b,c)\) with \(a+b+c \equiv 0 \bmod 2\).
On the other hand, it 
is notoriously difficult to turn  soft equidistribution results 
 into effective finiteness results. Such problems have arisen
 frequently in the literature;  
we recall one of the most basic problems of this type now 
in~\S~\ref{sec:effective},
as a possible model for thinking about Conjecture~\ref{conjA}.

\subsection{Effective results for roots of unity} \label{sec:effective}
Let \(\zeta\) and \(\xi\) be roots of unity, and consider the sum
\begin{equation} \label{robinson}
\alpha = 1 + \zeta + \xi.
\end{equation}
The element \(\alpha\) lies in the image of the torsion points of \(\GG^2_m\)
under the map to \(\mathbf{A}^1\) given by
\(1+X+Y\).
If one defines the \emph{house} of \(\alpha\) to be:
\[\ho{\alpha}:= \max_{\sigma} \left| 1 + \sigma \zeta + \sigma \xi \right|\]
as \(\sigma\) ranges over all automorphisms \(\sigma \in \Gal(\Qbar/\Q)\),
then the triangle inequality gives \(\ho{\alpha} \le 3\).
What can we say about the possible \(\alpha\) if we restrict
to those for which \(\ho{\alpha} < B\) for some explicit \(B < 3\)?
The equidistribution results of the flavor employed in~\cite{Bilu} show that,
for any fixed \(B < 3\), all such \(\alpha\) lie on the image of a finite union
of translates by torsion points of proper subgroup varieties of \(\GG^2_m\),
which (in this case) will be of dimension one and zero.  Can one make
this explicit?
An classical argument due to  Kronecker~\cite{MR1578994} 
shows that if \(\ho{\alpha} \le 2\),
then \(\alpha\) is actually a sum of at most two roots of unity,
and subsequently turns out to be in the image of the  the cosets
\((X,\rho X) \subset \GG^2_m\) of the diagonal
where \(\rho\) is any root of \(\rho^3 = -1\).
In~\cite{Robinson},
Raphael Robinson raised this very question of determining all \(\alpha = 1 + \zeta + \xi\)
such that 
\[\ho{\alpha} \le \sqrt{5}.\]
The reason for  Robinson's choice of this particular bound will be more apparent below.
The soft argument combined with some geometry shows 
that --- with finitely many exceptions --- all such examples either have
\(\ho{\alpha} < 2\) or come from the coset \((X,-X^{-1}) \subset \GG^2_m\)
of the anti-diagonal, 
that is, coming from \(\alpha\) of the form
\begin{equation}
\label{family}
\alpha = 1 + \zeta - \zeta^{-1}.
\end{equation}
One can check that all specializations on the family~\eqref{family} 
indeed have \(\ho{\alpha} \le \sqrt{5}\), and this realizes
 \(\sqrt{5}\) as a  limit point of \(\ho{\alpha}\) for
 sums of three roots of unity; this is the second such limit point after
 \(2 = \sqrt{4}\) coming from numbers  which can also be expressed as
  sums of two roots of unity. However, determining the finitely
many zero dimensional exceptions is significantly harder ---
in this case they were ultimately determined by Jones~\cite{Jones} to consist
(up to some obvious equivalences)  to precisely
 the five following examples:
\[ \begin{aligned}
& 1 + \e(1/7) + \e(3/7),  1 + \e(1/13) + \e(4/13), 1 + \e(1/24) + \e(7/24), \\
& 1 + \e(1/30) + \e(12/30), 1 + \e(1/42) + \e(13/42), \end{aligned}\]
where \(\e(x) = \exp(2 \pi i x)\).
Jones
used a number of novel arguments via the geometry
of numbers and  Mahler's duality theorem to reduce the problem to a manageable
calculation  --- previous work of Schinzel and Davenport~\cite{MR205926}
had produced an explicit but totally unmanageable upper bound
(of the flavor that the roots of unity involved could be assumed to have
order less than \(10^{10}\), for example).

Let us now return to Conjecture~\ref{conjA}.
As explained in~\cite{WM}
(cf.~\cite[12.3.6]{Deligne}) The quaternion algebra \(B\)
splits at all real places for \((p,q,r)\) if and only if
there exists an integer \(k\) prime to \([2,p,q,r]\) such that \(d(k) \ge 0\),
where
\begin{equation}
\label{firstd}
\begin{aligned}
d(t)  &  = 4 - 4\cos^2( t \pi/p) - 4 \cos^2( t \pi/q) - 4 \cos^2(t \pi/r) \\
& \ \ \ - 8  \cos(t \pi/p) \cos(t \pi/q)  \cos(t \pi/r) \\
& = 4  \cdot \det \left| \begin{matrix} 1 & - \cos(\pi/p) & - \cos(\pi/q) \\
-\cos(\pi/p) & 1 & -\cos(\pi/r) \\
-\cos(\pi/q) & -\cos(\pi/r) & 1 \end{matrix} \right|.
\end{aligned}
\end{equation}
More concretely, a Hilbert symbol
\(\displaystyle{\left(\frac{\alpha,\beta}{K}\right)}\) for \(B/K\) is given
explicitly  by 
\[\begin{aligned}
\alpha &  = 4 - 4\cos^2( \pi/p) - 4 \cos^2(\pi/q) - 4 \cos^2(\pi/r)
- 8  \cos(\pi/p) \cos(\pi/q)  \cos(\pi/r)\\
 \beta & =  4 \cos^2 \pi/p - 4. \end{aligned}
\]
(See~\cite[Thm~2]{WM}.)
It is easy to recognize~\eqref{firstd} as simply the expression
\begin{equation}
\label{robinsontwo}
\alpha = -(\zeta + \zeta^{-1})(\xi + \xi^{-1})(\theta + \theta^{-1})
(\zeta \xi \theta + \zeta^{-1} \xi^{-1} \theta^{-1})
\end{equation}
for three roots of unity \(\zeta,\xi,\theta\) depending
in an elementary way on \(p\),  \(q\), and \(r\),
which ``reduces''
Conjecture~\ref{conjA} to a problem of the same flavor 
as the problem of~\cite{Robinson}
discussed above. Concretely,
the conjugates of \(\alpha\) lie \emph{a priori}
in the interval~$[-16,4] \subset \R$, and so we want to classify
all \(\alpha\) such that
\[\ho{8 + \alpha} \le 8.\]
  The increased apparent  level
of complication of \eqref{robinsontwo} over~\eqref{robinson}
suggests that finding an explicit bound which reduces
Conjecture~\ref{conjA} to a  manageable computation
may present some difficulties. At the same time, there are a number
of further approaches in the literature for studying such questions.
As noted, there is the work of Jones in~\cite{Jones}
and in subsequent papers~\cite{Jones1,Jones2,Jones3,Jones4} as well as
the work of Cassels~\cite{Cassels}, as well as more Fourier--theoretic
approaches such as Davenport--Schinzel~\cite{MR205926}.

\subsection{The strategy}

In spirit, our argument is much closer to that of~\cite{WM}
than anything in~\cite{Bilu}.
We now discuss the general ideas behind our approach.
Suppose one wants to study the distribution of the
values of polynomial
in a single variable evaluated at roots of unity, but in an effective manner
(this is already interesting for \(f(X) = X\)).
The soft statement in this case is the fact that, for \(m\) large,
the conjugates of \(\zeta = \e(1/m)\) are equidistributed along the unit circle.
More precisely, however, they have the form \(\e(k/m)\) with \((k,m) = 1\).
So to understand how close to a point on \(|z| = 1\) one can find
such an \(\e(k/m)\)  for any \(m\), one has
to understand how close to any  real \(t \in \R\) one can choose an
integer \(k\) with \((k,m) = 1\). This immediately reduces to the problem
of understanding how long an arithmetic progression \(a,a+1,a+2,\ldots\)
has to be before there exists an element in this sequence coprime to \(m\).
By definition, this is given by the Jacobsthal function \(J(m)\).
Suppose that \(m\) has  \(r\) distinct prime factors.
A theorem of Iwaniec~\cite{Iwaniec} shows that \(J(m) \ll (r \log r)^2\), but Iwaniec's
results are not effective. A theorem of Kanold~\cite{Kanold}
shows that  \(J(m) \le 2^r\).
This is definitely effective but it is not optimal --- if \(m =pqr\), then an elementary
argument shows that \(J(pqr) \le 6\),
and if \(m = pqr\kern-0.1em{s}\), then \(J(m) \le 10\). It is crucial for our ultimate applications
that we have excellent bounds on \(J(m)\) for \(m\) with a moderately
large number of prime factors, say at most \(20\) prime factors. Fortunately there are 
results in the literature  which give such bounds. For example, if \(m\)
has \(20\) distinct prime divisors, then \(J(m) \le 174\) (see Lemma~\ref{better}), which is far
better than Kanold's bound \(J(m) \le 1048576\). We recall
the results we need about the Jacobsthal function in~\S~\ref{sec:jacob}.

Now let us return to higher dimensions. Recall that we wish
to find an integer \(k\) prime to \([2,p,q,r]\) so that \(d(k) \ge 0\) where
\[d(t) = 4 - 4\cos^2( t \pi/p) - 4 \cos^2( t \pi/q) - 4 \cos^2(t \pi/r)
- 8  \cos(t \pi/p) \cos(t \pi/q)  \cos(t \pi/r).\]
Let \(\Lambda \subset \Z^3\) denote the sub-lattice of index two
with \(a+b+c \equiv 0 \bmod 2\), and for a vector \(v  = (x,y,z) \in \R^3\) we let
\[|\v - \Lambda| :=  \min \left( |x-a| + |y-b| + |z-c|\right),
 \ (a,b,c) \in \Lambda.\]
If
\(\displaystyle{\v  = \left(\frac{1}{p},\frac{1}{q},\frac{1}{r} \right)}\), 
then the condition that \(d(t) \ge 0\) is equivalent
 to the condition that
\begin{equation}
\label{easier}
 |t \v - \Lambda| \ge 1.
 \end{equation}
This follows from~\cite[Thm~1.10]{Bilu}.
Instead of trying to prove that \(d(k) \ge 0\) or \(|k \v - \Lambda| \ge 1\)
for some integer \(k\) prime to \([2,p,q,r]\), we could first ask to find
any integer \(k\) with this property,  or --- what is nearly equivalent
for big \(p,q,r\) --- a \emph{real} number \(t\) so that \(d(t) \ge 0\).
Even better, we can try to  produce a \(t\) for which
\(|t \v - \Lambda| \ge 1 + \varepsilon\)
for some explicit \(\varepsilon > 0\). 
Then by varying \(t\)  slightly, one can hope to find an integer \(k\)
close by to \(t\) for which \(k\) is prime to \([2,p,q,r]\) and for which 
\[ 1 + \varepsilon < |t \v - \Lambda| \sim |k \v -\Lambda|  \]
 does not change that much and so the latter is at least one.
Our key technical result (Theorem~\ref{5lemma}) is that there
always exists
a \(t \in \R\) so that
\[ \left|  \left(\frac{t}{p},\frac{t}{q}, \frac{t}{r} \right) - \Lambda \right| \ge
1 + \frac{1}{5}.\]
(This result is best possible, see Remark~\ref{improved}.)
We shall give two methods to prove this theorem.
The first, using Fourier analysis,
constitutes most of~\S~\ref{sec:fourier}.
The  second,  using geometry of numbers, is
discussed in~\S~\ref{sec:minkowski}.
We only carry out the first approach --- the second one certainly
works in principle but it is not a priori clear how computationally
feasible it might be.

Returning to our argument, once we have found a \(t \in \R\) so that
\(|t \v - \Lambda| - 1 > 0\) is big, we could hope to vary \(t\) slightly to make
it an integer prime to \(2pqr\). This works, proving that that
\[ \frac{1}{p} + \frac{1}{q} + \frac{1}{r}\]
is relatively small compared to the lcm \(n = [2,p,q,r]\) ---
 but it fails otherwise, in particular if one of \(p\), \(q\), or \(r\) is small.
 So we need a separate argument to deal with the
 cases when \(1/p+1/q+1/r\) is small compared to \(n\) and when
 \(\min(p,q,r)\)
 is small compared to \(n\). 
 The first case is carried out in~\S~\ref{sec:nbig}, and the case
 when~\(\min(p,q,r)\) is small is carried out in~\S~\ref{sec:small}.
 It turns out that this last case is ultimately the most computationally intensive.
 Finally, in~\S~\ref{sec:resolution} we show how the cases
 understood in~\S\S~\ref{sec:nbig}, ~\ref{sec:small} cover all cases.
 
A basic trick that we employ frequently is the following.
Suppose that \(p\) is divisible by a prime which does
not divide \(q\) or \(r\), or more generally that \(p\) is divisible by a strictly
higher power of a prime than \(q\) or \(r\), or perhaps a product of such primes.
Then we may  freely
conjugate one of the roots of unity while keeping the other
two roots fixed. This is clearly advantageous for our purposes.
In practice we employ this idea in the setting of vectors in \(\R^3/\Lambda\)
rather than roots of unity, but the idea is the same. We explain
this in~\S~\ref{sec:twisting}. 
More generally, in~\S~\ref{sec:lower}, we prove some analogues
of our main theorem in dimensions one and two, which we use
inductively to study the problem in dimension three.

Finally, there are a few tricks which by trial and error we simply found
that helped --- they involve some combinations of ideas that relate
both the idea of Galois conjugates and some elementary geometry
of  lines in \(\R^2/\Z^2\) of rational slope --- we apologize that some
of these will ultimately seem somewhat ad hoc.

\subsection{Some remarks about the title}
Given some object~$X$ defined over~$\Kbar$, the term ``field of moduli''
usually refers to the smallest field~$L$ for which there exists
an automorphism~$\psi_{\sigma}: X^{\sigma} \simeq X$
for every~$\sigma \in \Gal(\Kbar/L)$, whereas a ``field of definition''
is a field~$L$ so that~$X$ has a model over~$L$.
The reason that fields of moduli are not always  fields of definition
is because the~$\psi_{\sigma}$
do not always give compatible descent data due to the existence
of automorphisms of~$X$.
In our context, we think of~$X$ as the (unique up to conjugacy) representation
of the triangle group~$\Delta$  to~$\PSL_2(\R)$. The field of moduli in this case is the
trace field~$K$, and a field of definition~$L/K$ is any field where the representation
admits an explicit matrix model over this field. We hope that this terminology is clear
(although we only ever use this terminology in the title and in this subsection).

%
%

\subsection{Preliminaries}

For  integers \(n_i\) as \(i=1,\ldots,r\), let \([n_1,n_2, \ldots n_r]\) denote the lowest common multiple of the \(n_i\).
Let~\((p,q,r)\) be  integers, let~\(n = [2,p,q,r]\) be the least common multiple of these numbers, which is
equal to \([p,q,r]\) if at least one of~\(p\), \(q\), \(r\) is even, and twice this otherwise.

\begin{df} Let~\(\Lambda \subset \Z^3\) be the lattice consisting of triples of integers~\((a,b,c)\) with~\(a+b+c\)
even.
\end{df}

Given two vectors~\(\v\) and~\(\w\) in~\(\R^3\), we define a distance~\(| \cdot |_{1}\) to be the sum of the absolute
values of the differences, i.e.:
\[|\v - \w|_{1}:= |v_1 - w_1| + |v_2 - w_2| + |v_3 - w_3|.\]
Since we only use this distance function, we simply write \(|\v|\).

\begin{df} \label{notation}
Given a vector~\(\v \in \R^3\), we define the distance of~\(\v\) to the lattice~\(\Lambda\) to be the smallest distance from~\(\v\) to any
point in~\(\Lambda\), i.e.,
\[ | \v - \Lambda| := \min\{ |\v - \lambda| \, ; \, \lambda \in \Lambda\}\]
More generally, for \(\v \in \R^n\) and any lattice \(\Lambda \subset \R^n\),
we make the same definition, still using the \(| \cdot |_1\) norm on \(\R^n\).
\end{df}

 Note that this is a genuine distance function and so satisfies
the triangle inequality
and its variants. Apart from the case of our lattice \(\Lambda\),
we will most often use this notation for \(x \in \R\) and \(\Lambda = \Z\)
or \(2\Z\) in \(\R\), but we shall also use it for the standard
lattice \(\Z^3\) in \(\R^3\), particularly in ~\S~\ref{sec:fourier}.

\begin{example}  If~\(\w = (1/2,1/2,1/2)\), then~\(|\w - \Lambda| = 3/2\). 
This is the maximal value of \(|\v - \Lambda|\).
\end{example}

We now rephrase the basic result of this paper in elementary terms.
By~\cite{Bilu},  this suffices to prove
Conjecture~\ref{conjA}.

\begin{theorem} \label{maintheorem} Let \((p,q,r)\) be a triple with \(1/p+1/q+1/r < 1\)
which is not in the Hilbert Series. Let \(n = [2,p,q,r]\). Then there exists an integer~\((k,n)=1\)
such that~\(|k \v - \Lambda| \ge 1\), where
\[\v = \left( \frac{k}{p}, \frac{k}{q}, \frac{k}{r} \right)\]
\end{theorem}

We finish this section with some more preparatory lemmas.
\begin{lemma} \label{extra} Fix \(x\) and \(y\), and let
\(M(x,y) = \max \left| \left(x,y,z\right) - \Lambda \right|   \) as \(z\) ranges over
elements of \(\R\).
Then
\[
\begin{aligned}
M(x,y) &  = M(y,x) \\
M(x,y)  & = M(x \bmod 1,y \bmod 1) \\
M(x,y) & = M(1-x,y) \end{aligned}
\]
\end{lemma}

\begin{proof}
Since \(\Lambda\) is invariant under permuting the first two entries,
the first claim follows. For the second, we can replace \(z\) by \(z \pm  (1,0,1)\)
or \(z \pm (0,1,1)\). This brings us to the third claim.
The main point is that, for \((a,b,c) \in \Lambda\) we have \((1-a,b,1-c) \in \Lambda\),
and then
\[ |(x,y,z) - (a,b,c)| = |x-a|+|y-b|+|z-c| = |(1-x,y,1-z) - (1-a,b,1-c)|.\]
In particular, the distances from \((x,y,z)\) to the set of lattice
points is the same as the distances from \((1-x,y,1-z)\) to the set of lattice
points. But then
\[M(x,y) = \max \left| \left(1-x,y,1-z\right) - \Lambda \right|
= \max \left| \left(1-x,y,z\right) - \Lambda \right| = M(1-x,y).\]
\end{proof}

\begin{lemma} \label{computingM} If \(0 \le x,y \le 1/2\), then
\(M(x,y) = 1+\min(x,y).\) More generally,
\[M(x,y) = 1 + \min(|x-\Z|,|y-\Z|)\]
\end{lemma}

\begin{proof} To see how the latter equality follows
from the former, note that \(|t - \Z| = |1-t-\Z|\) and so
by Lemma~\ref{extra} we can reduce to this case.

 Suppose that \((a,b,c)\) is the closest lattice
point to \((x,y,z)\). 
We have
\[ | (x,y,z) - (a,b,c)| = |x-a| + |y-b| + |z-c|.\]
Replacing \(a\) by \(-a\) or \(b\) by \(-b\) preserves 
the property of \((a,b,c) \in \Lambda\), but increases the RHS,
so without loss of generality we have
\[ (a,b) \in \{(0,0),(1,0),(0,1)\}\]
At the same time, we can always take \(z \in [0,2]\) and
so either \(c=0,2\) or \(c = 1\) is optimal.
Thus
\(|(x,y,z) -\Lambda|\) is the minimum of the four quantities:
\[
x + y + z,x+y+2-z, 1-x+y + |z-1|, x + 1-y + |z-1|.\]
Suppose that \(0 \le x \le y\). Then with the choice \(z = 1-x\),
Suppose without loss of generality that \(x \le y\). Then
with the choice \(z = 1-y\), these quantities become
\[1+x, 1+x + 2y, 1+x + 2(y-x), 1+x,\]
which has minimum \(1+x\) and shows that \(M(x,y) \ge 1+x\).
On the other hand, note that
\[
\begin{aligned}
x+y+z \le 1+x, &\quad \text{for} \ z \le 1-y, \\
x + 1-y + |z-1| \le 1+x, & \quad \text{for} \ 1-y \le z \le 1+y, \\
x + y + 2 -z \le 1+x, & \quad \text{for} \  z \ge 1+y, \\
\end{aligned}
\]
which shows that \(1+x\) is optimal.
 \end{proof}

\begin{lemma} \label{weyl}
The norm \(|\v|\) is invariant under the group
\[(\Z/2 \Z)^3 \ltimes S_3 = \Z/2\Z \wr S_3 \subset \mathrm{O}(3)\]
acting on the coordinates
by permutation and by signs.
If \(\w = (1/2,1/2,1/2)\), then \(|\v - \w - \Z^3|\) is also
preserved by this group, whereas
The distances
\[ |\v -\w - \Lambda|\]
is  preserved only by
 \[(\Z/2\Z \wr S_3) \cap \mathrm{SO}(3) \simeq S_4.\]
\end{lemma}

\begin{proof} It suffices to observe that these groups
preserve the lattice \(\Z^3\) and \(\Lambda\), as well as
\(\w + \Z^3\). Since \(2 \w \notin \Lambda\), the
orbit of this group on \(\w + \Lambda\) is given by
\(\pm \w + \Lambda\) and the stabilizer has index two.
\end{proof}

\section{Fourier Analysis}
\label{sec:fourier}
The main result of this section is the following:

\begin{ithm} \label{5lemma} Let~$(p,q,r)$ be non-zero rational
numbers. Then there exists a real number $t$ such that
\[ \left| \left( \frac{t}{p}, \frac{t}{q}, \frac{t}{r} \right) - \Lambda \right| \ge 
1 + \frac{1}{5}.\]
\end{ithm}

\begin{remark} \label{improved}
The constant is best possible; one can do no
better than equality in the case of \((p,q,r) = (1,2,6)\), \((2,3,6)\),
and \((3,4,12)\) --- see Figure~\ref{fig:3Db}.
On the other hand, 
our proof shows that, apart from these exceptions,
the stronger inequality holds:
\[ \left| \left( \frac{t}{p}, \frac{t}{q}, \frac{t}{r} \right) - \Lambda \right| \ge 
\frac{3}{2} - \frac{
\displaystyle{3 \arccos \left( \frac{3^{1/6}}{2^{1/3}}\right)}
}{\pi} =  1.206646\ldots > 1 + \frac{1}{5}.\]
 Our proof in principle could be modified to increase
the constant further (still with finitely many exceptions),
but the gain in utility would be
marginal.  
\end{remark}

\subsection{The idea of the proof}
Note that by scaling, we may assume that \((p,q,r)\) are all 
non-zero integers. 
Moreover, using the action of~\(\Z/2\Z \wr S_3\)
as in Lemma~\ref{weyl} we are also free to change any of the signs
of \(p,q,r\) (note that the element $-I$ sends any line through the origin
to itself).
Our approach is as follows.
For most \((p,q,r)\), we will actually prove  the  statement
that there exists a real \(t\) so that
\begin{equation}
\label{weproveit}
\left| \left( \frac{t}{p}, \frac{t}{q}, \frac{t}{r} \right) - \Z^3 \right| \ge 
1 + \frac{1}{5},
\end{equation}
where we recall that \(\Lambda\) has index two in \(\Z^3\)
so that~\eqref{weproveit} is a stronger statement.
More precisely, we can even show such an inequality
with \(1+1/5\) replaced by \(1 + 1/2 - \varepsilon\)
for any \(\varepsilon > 0\) as long as we 
\emph{avoid} finitely many hyperplanes
\[\frac{a}{p} + \frac{b}{q} + \frac{c}{r} = 0,\]
where ``finitely many'' will  depend on \(\varepsilon\)
in  a way that one can in practice quantify.
There are two approaches we have to proving
such inequalities; one inspired by Fourier analysis
(which is what we ultimately use) and a second using
the geometry of numbers which we discuss in the section
below.
The Fourier analysis approach is roughly as follows.
We are working in the settings of functions on the compact
torus \(\R^3/\Z^3\) or \(\R^3/\Lambda\) or \(\R^3/(2 \Z)^3\).
Let \(L \subset \R^3/\Lambda\) be a rational line, which is an embedded
one-dimensional torus.
Let \(\chi(x,y,z)\) be the characteristic function of the region
where we want \(L\) to intersect.
To show the intersection is non-zero, we want to compute
\[ \int_{L} \chi(t) dt\]
and prove that it is non-zero.
This   leads to a soft proof that the result
holds away from finitely many hyperplanes; arguments
of this kind are well-known to the experts.
However, we need a result which is valid in \emph{all} cases,
so in particular we need to identify explicitly the list
of exceptional hyperplanes which then need to be considered.
Moreover, the particular region we are interested in is really quite
complicated, as seen in Figure~\ref{fig:3D}.
\begin{center}
\begin{figure}[H]
  \includegraphics[width=6cm]{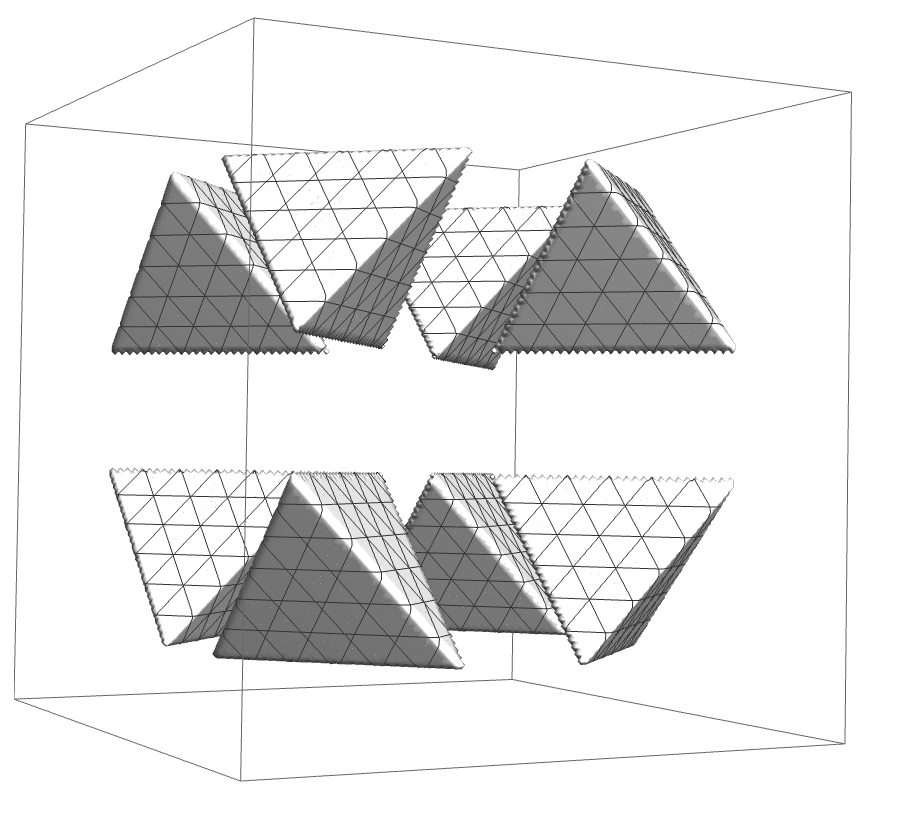}
  \caption{The  region \(|(x,y,z) - \Lambda| \ge 6/5\)
  in \(\R^3/(2 \Z)^3\).}
  \label{fig:3D}
  \end{figure}
\end{center}
As a result, the Fourier series of \(\chi(x,y,z)\) is a complicated mess.
For example, for the region \(|\v - \w| \ge 1 + 1/5\) in \(\R^3/\Z^3\),
after the constant term \(9/250\), the coefficient
of \(e^{2 \pi i x}\) is
\[ \frac{
\displaystyle{
5 \sqrt{2 (5 + \sqrt{5})} - 12 \pi}
}
{20 \pi^3}
\]
and things only get worse from 
there.\footnote{In contrast, the characteristic function~$\chi$
of the region
\(|\v - \Lambda| \ge 1\) is  more pleasant. A current project
by undergraduates at the University of Chicago 
analyzing~$\chi$ more closely,  in order to exactly determine
the smallest limit point of the \emph{ramification
density} (see~\cite{Bilu}) which is some real number \(r < 0\),
and perhaps even the smallest non-zero
ramification density \(\rho_H > 0\).}

One alternative  is to use a more convenient function
that the characteristic function. There \emph{is} a function
which is transparently given to us in this situation, namely
\[\begin{aligned}
d(x,y,z) & = 4 - 4 \cos^2(\pi x) - 4 \cos^2(\pi y)
- 4 \cos^2(\pi z) - 8 \cos(\pi x) \cos(\pi y) \cos(\pi z), \\
d(t) & =  d\left(\frac{t}{p},\frac{t}{q},\frac{t}{r}\right).
\end{aligned}
\]
This function is  non-negative \emph{precisely}
on the region 
where we wish to ultimately find \emph{integer} points. 
(One way to translate between the roots of unity
picture and the lattice picture is given by~\cite[Thm~2]{WM},
but it is also transparent by~\eqref{robinsontwo}.)
However,
this function has the disadvantage that while it is positive
where we want it to be positive, it is also (even more) 
negative elsewhere, which complicates the analysis.
The situation is not hopeless, and indeed
 using \(d(t)\) we initially proved a version
 of Theorem~\ref{5lemma} with \(1+1/5\) replaced
 by \(1+1/13\), but the improvement to \(1/5\)
 will ultimately be important from a computational
 viewpoint.
 
 Our approach is to replace \(d(x,y,z)\)
 by the function
 \[\begin{aligned}
e(x,y,z) & = (8 \sin(\pi x) \sin(\pi y) \sin(\pi z))^2 \\
e(t) & = \left(8 \sin( \pi x/p) \sin(\pi y/q) \sin(\pi z/r)\right)^2.
\end{aligned}
\]
This function is now invariant under the larger lattice \(\Z^3\).
This function also has the advantage that it is \emph{non-negative}
everywhere and large exactly near the point
\((1/2,1/2,1/2)\),
so \(e(x,y,z)\) or rather its powers  play a convenient
proxy role for the characteristic function. Actually,
instead of powers of \(e(t)\), we  shall  consider
a minor variant --- powers of \((e(t) - 24)\) --- which asymptotically  is similar
but in practice  gives better results within the limits of our computations.

Note that Theorem~\ref{5lemma} is not true
if one replaces  \(\Lambda\) by \(\Z^3\). In  fact,
this fails not only finitely often but for all
points on the hyperplane
\[\frac{1}{p} + \frac{1}{q} + \frac{1}{r} = 0.\]
Naturally in our proof
we are forced to consider this hyperplane
separately and return to the lattice \(\Lambda\).
Moreover, we also
have to deal with a moderate number of explicit
triples \((p,q,r)\)  --- say on the order of a million ---  for those also we check
the less restrictive condition with \(\Lambda\)
rather than \(\Z^3\) as well.

\subsection{The start of the proof of Theorem~\ref{5lemma}}

\begin{proof}
By scaling, we may assume that \((p,q,r)\) are all 
non-zero integers. 
We introduce the quantity
\[e(t) = \left(8 \sin( \pi x/p) \sin(\pi y/q) \sin(\pi z/r)\right)^2.\]
We choose this function for a few reasons.
First,
the function \(8 \sin(\pi x) \sin(\pi y) \sin(\pi z)\) in
 obtains its global maximum
exactly at the points \((1/2,1/2,1/2) \bmod \Z^3\). Hence if we can show
that it has a large average over the line \((t/p,t/q,t/r)\), we can
hopefully show that this point is close  to \((1/2,1/2,1/2)\).
Second, the trigonometric form of \(e(t)\) means that we can indeed
compute the relevant integrals.
Clearly~$e(t)$ is periodic with period (dividing) \(pqr\).
Moreover, it is non-negative, and there is a trivial upper and lower bound
of \(0 \le e(t) \le 64\). 
We would like to find values of \(t\) for which \(e(t)\) it is as large as possible.
One way is to investigate the moments.
More precisely, we shall consider
\begin{equation}
\label{average}
E[(e(t) - 24)^{m}] = \frac{1}{pqr} \int_{0}^{pqr} \left(e(t) - 24\right)^{2m} dt,
\end{equation}
and then use the  bound:
\begin{equation}
\label{zero}
\max(|e(t) - 24|) \ge \sqrt[m]{E[(e(t) - 24)^{m}]}.
\end{equation}
Since \(e(t) \ge 0\) and is non-constant, if  the RHS of~\eqref{zero} is at least \(24\),
it gives a lower bound for \(e(t)-24\) 
for some point \(t \in \R\) rather than its absolute value.
However, we can compute the quantity in~\eqref{average}
by a simple integration. Certainly, using the exponential formula
for \(\sin \pi x\), we can write, with the sum over \(\lambda = (a,b,c) \in \Z^3\),
and with
\(\displaystyle{\v =  \left(\frac{1}{p},\frac{1}{q},\frac{1}{r} \right)}\),
\[\begin{aligned}
(e(t) -24)^m & = \sum_{\lambda \in \Z^3} C_{\lambda,m} 
\exp \left(2 \pi i \lambda \cdot \v \cdot t  \right) \\
 & = \sum_{(a,b,c) \in \Z^3} C_{(a,b,c),m} 
\exp \left( 2 \pi i \left( \frac{a}{p} + \frac{b}{q} + \frac{c}{r} \right) t \right),
\end{aligned}\]
where \(C_{\lambda,m} = 0\) for all but finitely many~\(\lambda\).
We plainly have
\[E[(e(t)-24)^m] = \sum_{\lambda \in \Z^3} C_{(a,b,c),m} \times
\begin{cases} 1 &  \lambda \cdot \v = 0, \\
0 & \text{otherwise}. \end{cases}
\]
In particular, the possible values of this average are as follows:
\begin{enumerate}
\item {\bf The generic case:\rm}  \(\lambda \cdot \v \ne 0\) for all \(\lambda \ne 0\)
with \(C_{\lambda,m} \ne 0\), in which case 
\[E[(e(t)-24)^m] = C_{0,m}.\]
\item  {\bf The co-dimension one case:\rm} The \(\phi\) for which  \(\phi \cdot \v = 0\) for some \(\phi \ne 0\)
with \(C_{\phi,m} \ne 0\) generate a one dimensional subspace of \(\R^3\), in which case
\[E[(e(t)-24)^m] = \sum_{\lambda \in \phi \Q \cap \Z^3} C_{\lambda,m}\]
where the sum is over rational multiples of (any such) \(\phi\).
\item  {\bf The co-dimension two  case:\rm} We have \(\phi \cdot \v = 0\) and \(\psi \cdot \v = 0\)
and \(C_{\phi,m}, C_{\psi,m} \ne 0\) where \(\phi\) and \(\psi\) are linearly
independent, in which case
\[E[(e(t)-24)^m] = \sum_{\lambda \in \phi \Q  \oplus \psi \Q \cap \Z^3} C_{\lambda,m}\]
where the sum is over rational multiples of \(\phi\) and \(\psi\) in \(\Z^3\).
\end{enumerate}
Since \(\v \ne 0\),  it cannot be orthogonal to  three linearly
independent vectors,  and hence these are the only possibilities.
Clearly if there are only finitely many \(C_{\lambda,m}\), there are finitely
many values of the sum, and in theory we can compute
all the possible values.
In order to do so in a convenient way, observe that, for \(m=1\), we have
the following values:
\[
\begin{aligned}
C_{(0,0,0),1} & =  -16, \\
C_{(\pm 1,0,0),1}  = C_{(0,\pm 1,0),1} = C_{(0,0,\pm 1),1} & = -4, \\
C_{(\pm 1,\pm 1,0),1}  = C_{(\pm 1,0,\pm 1),1} = C_{(0,\pm 1,\pm 1),1} & =2, \\
C_{(\pm 1, \pm 1, \pm 1),1} & = -1, \\
\end{aligned}
\]
and all other values are zero, and then the inductive formula:
\[C_{\lambda,m} = \sum_{\lambda' + \lambda'' = \lambda} C_{\lambda',m-1} C_{\lambda'',1},\]
together with the remark that \(C_{\lambda,m} = 0\) unless all the entries of \(\lambda\)
are bounded in absolute value by \(2m\), and the sum only needs to be evaluated
over the \(1 + 6 + 24 + 8 = (1+2)^3 = 27\) terms with \(C_{\lambda'',1} \ne 0\).

We now turn to computing the value of \(E[(e(t)-24)^{12}]\).

\begin{remark}
Originally we  computed the moments \(E[(e(t)-24)^{m}]\),
but the current method gives a better bound as soon as it gives a non-trivial
bound.   In fact, the optimal choice of \(\theta\) for
obtaining a bound for \(e(t)\) by considering twelfth powers
of the form \(E[(e(t)-\theta)^{12}]\)
is given by \(\theta \sim 24.5686 \ldots \) where \(\theta\) is
a root of a degree \(12\)  irreducible
polynomial in \(\Q[x]\).
The choice of the exponent \(12\) is close to the limit
of what one can compute with these direct computations.
 If necessary, we could
probably push this computation slightly further.
\end{remark}

\subsection{The generic case}
We find that 
\begin{equation} \label{value}
C_{0,12} =48938065973953984 \sim 4.8 \times 10^{16}.\end{equation}

\subsection{The co-dimension one case}
We are assuming that
the \(\phi\) for which  \(\phi \cdot \v = 0\) for some \(\phi \ne 0\)
with \(C_{\phi,m} \ne 0\) define a one dimensional subspace.
There are exactly \(24389\) vectors \(\lambda\) with \(C_{\lambda,12} \ne 0\),
which lie on \(6337\) different lines. In other words, our vector
\(\v\) has to lie on one of \(6337\) different hyperplanes, or else
\(E[(e(t)-24)^{m}]\) is given by the value in the generic case.
On these \(6337\) lines generated by vectors \(\phi\), the possible sums
\[\sum_{\lambda \in \phi \Q \cap \Z^3} C_{\lambda,m}\]
take on \(334\) different values, which, in increasing order, are given by:
\[
\begin{aligned}
& 14495307580935536, 14993075676088944, \\ & 16558274382015248, 17182507338527490, \\
& 17779880901663312, 20880831907741248, \\ &  21658015136699472, 23695946550006558, \\
& 25723702367996064, 
29439428154585408, \\ & 31449612130791616, 32684658530437488, \\
& 32786111957612896, 35548827082706688, \\ &  35574385398832360,
36520347436056576, \ldots 
\end{aligned}
\]
The vectors for which this quantity is strictly less
than the sixteenth term:
\begin{equation}
\label{smaller}
24^{12} = 36520347436056576  \sim 3.6 \times 10^{15}
< C_{0,12}
\end{equation}
 correspond to hyperplanes for which our
 computation of \(E[(e(t)-24)^{12}]\) does not give a 
 useful lower bound for the maximum of \(e(t)\) along this hyperplane.
Changing the signs of \(p,q,r\) appropriately
we may assume that \(a,b,c\) are non-negative.
The possible triples are below, along with a point \(P\)
on the hyperplane chosen to be close to \((1/2,1/2,1/2)\)
(in practice we tried to choose an optimal such point but
there is no need to prove we were successful):
  \begin{center}
\begin{tabular}{*4c}
 \multicolumn{4}{c}{\emph{Hyperplanes with \(E[(e(t)-24)^{12}] < 24^{12}\)}}\\
 \toprule
 triple & $(a,b,c)$ & $P$  & $|P - (1/2,1/2,1/2)|$ \\
\midrule
$1$ & $(0,0,1)$ & $(1/2,1/2,0)$ & $1/2$ \\  
$2$ & $(1,1,1)$ & $(1/3,1/3,1/3)$ & $1/2$ \\  
$3$ & $(0,1,2)$ & $(1/2,1/2,1/4)$ & $1/4$ \\  
$4$ & $(1,2,2)$ & $(1/2,1/2,1/4)$ & $1/4$ \\  
$5$ & $(0,1,4)$ & $(1/2,2/5,2/5)$ & $1/5$ \\  
$6$ & $(0,2,3)$ & $(1/2,2/5,2/5)$ & $1/5$ \\  
$7$ & $(1,1,3)$ & $(1/2,1/2,1/3)$ & $1/6$ \\  
$8$ & $(1,3,3)$ & $(1/2,1/2,1/3)$ & $1/6$ \\  
$9$ & $(2,2,3)$ & $(1/2,1/2,1/3)$ & $1/6$ \\  
$10$ & $(0,2,5)$ & $(1/2,3/7,3/7)$ & $1/7$ \\  
$11$ & $(0,3,4)$ & $(1/2,3/7,3/7)$ & $1/7$ \\  
$12$ & $(1,2,4)$ & $(1/2,1/2,5/8)$ & $1/8$ \\  
$13$ & $(1,4,4)$ & $(1/2,1/2,3/8)$ & $1/8$ \\  
$14$ & $(2,3,4)$ & $(1/2,1/2,3/8)$ & $1/8$ \\  
$15$ & $(2,2,5)$ & $(1/2,1/2,2/5)$ & $1/10$ \\  
$16$ & $(1,1,5)$ & $(1/2,1/2,2/5)$ & $1/10$ \\
\bottomrule
\end{tabular}
\end{center}
We shall now prove Theorem~\ref{5lemma}
explicitly in these sixteen classes cases, so that later
we will be free to assume that \(E[(e(t)-24)^{12}]\)
is bounded below by~\eqref{smaller}.
Note we are using Lemma~\ref{weyl} to see we are free to choose the signs of the hyperplane coefficients.
The case \((0,0,1)\) does not arise since it would force \(r = 0\).
The case \((1,1,1)\)  also requires exceptional treatment so we leave
it until the end. In all other cases,  there exists a point \(P\) on the hypersurface
with
\[|P - \w| \le \frac{1}{4}.\]
with \(\w = (1/2,1/2,1/2)\).
Thus it remains to determine all the lines \(L\) on \(H\) such that
\[|L - H| \le \left(|\w| - 1 - \frac{1}{5} \right) - \frac{1}{4} = \frac{1}{20}.\]
As usual,
the notation here means the minimum distance
of any point on \(L\) 
to any point on \(H\) in the \(| \cdot |_1\) norm.
The basic idea is that any line \(L\) on a
\emph{rank two} torus is easily seen to be qualitatively
close to any given point as soon as the slope
of the line is sufficiently large. In this way, we will
reduce the computation to finitely many lines which
we check using \texttt{magma}. We return to the proof.
Consider the hyperplane
\[H: a x  + b y + c z = 0 \subset \R^3/\Z^3\]
Without loss of generality, we can assume that \(c \ne 0\),
and that \((a,b,c)\) have no common factor.
More precisely, we shall assume that \(c\) is the largest
element of the triple.

\begin{lemma} \label{lemma:sur}
There is a finite surjection:
\[\pi: T:=\R/\Z \times \R/\Z \rightarrow H \subset \R^3/\Z^3\]
given by
\[(s,t) \rightarrow (cs,ct,-as - bt).\]
\end{lemma}

\begin{proof}  It is easy to see that the map is well-defined.
Let \((x,y,z)\) be a point on \(H\).
We can find
\(s,t\) so that \((s,t) \mapsto (x,y,z') \in H\), and then it follows that
 \(c(z-z') = 0 \in \R/\Z\), and so \(z-z'\) is a multiple of \((0,0,1/c)\).
 Since \(\pi\) is linear,  it suffices to show that its  image  contains
 the subgroup generated by this element.
 Let \(s = i/c\) and \(t = j/c\) for integers \(i\) and \(j\). Then the image under 
\(\pi\) will generate such a subgroup as long as \(ai + bj\) is prime to \(c\).
But we can find a choice of \(i\) and \(j\) so that \(ai+bj\)  is  equal to the greatest
common divisor \((a,b)\), and this is prime to \(c\) since we are assuming
that \((a,b,c)\) do not all have a common factor.
\end{proof}

We can also show that the map \(\pi\) will have degree \(c\),
but that is not relevant for our purposes;
 The only fact we need
is that any line on \(H\) through the origin is the image of a line \(L\)
on \(T\) through the origin (namely, \(\pi^{-1}(L)\)) which follows
directly from the fact that \(\pi\) is surjective by Lemma~\ref{lemma:sur}.
Such a line \(L\) has the form
\begin{equation}
\label{line}
(s,t) = (uz,vz), \ (u,v) = 1,
\end{equation}
we define the height \(h(L)\) of this line to
be \(h(L) = \max(|u|,|v|)\).

\begin{lemma} \label{elementary} Let \(P \in H\),where \(H\)
is one of the \(14\) hyperplanes numbered \(3\) to \(16\)  above. If the height of \(L\) satisfies
\(h(L) \ge 70\),
then there is a point \(Q \in \pi(L)\)  so that 
\[\displaystyle{|P -Q|  \le \frac{1}{20}}.\]
If \(h(L) \ge 84\), then this bound can be improved to 
\(\displaystyle{\frac{1}{24}}\).
\end{lemma}

\begin{proof}
Let \(\pi^{-1}(P)\) be (any) choice of pre-image of \(P\).
Let \(L\) be of the form~\eqref{line}, and assume that \(h(L)=u\).
We may certainly find a \(z \in \R\) so that the first coordinate
of \((uz,vz)\) is equal to the first coordinate of \(\pi^{-1}(P)\).
If we then replace \(z\) by \(z + 1/u\), we may vary
the second coordinate by \(v/u\) and then by repeating this and
using that \((u,v)=1\), we may find a point so that
\[ |R - \pi^{-1}P| \le \frac{1}{2h}.\]
and moreover that the first coordinate is zero. If \(h(L)=v\),
the exact same argument gives a point where the second coordinate is zero.
If we write \(R - \pi^{-1}(P) = (x,y)\), 
so \(xy = 0\) and \(\max(|x|,|y|) \le (2h)^{-1}\), we find that
\[
\begin{aligned}
|\pi(R) - P| & \le
c x 
+ c y
+ a x 
+ by \\
& = (a+c) x + (b+c) y \\
 & \le \frac{1}{2h} \max \left(a+c,b+c \right)
\end{aligned}
\]
But for the \(14\) hyperplanes above we find that
the maximum \(b+c\) and \(a+c\) is \(7\),
and we are assuming that \(h \ge 70\), and so \(2h \ge  140\),
and so
\[|\pi(R) - P| \le \frac{7}{140} = \frac{1}{20}.\]
Replacing \(70\) by \(84\) gives the improved bound.
\end{proof}

 As a consequence, for each of the \(14\) hyperplanes
 \(H\) we are currently
 considering, if we take any line on \(H\)
 whose pullback \(L\) to \(T\) has \(h(L) \ge 70\),
 then there exists a point \(Q = \pi(R)\) on \(L\)
 with:
 \[ |Q - P| \le \frac{1}{20}, \quad
 |P - \w| \le \frac{1}{4},\]
 and hence
 \[ \begin{aligned}
 |Q -\w| & \le \frac{1}{20} + \frac{1}{4} = \frac{3}{10}, \\
 |Q - \Z^3|  & \ge |\w - \Z^3| - |Q - \w| \ge \frac{3}{2}
 - \frac{3}{10} = 1 + \frac{1}{5}. \end{aligned}\]
 With the stricter condition \(h(L) \ge 84\), this
 bound improves to \(1 + 5/24\).
 For the remaining lines \(L\) of height \(h(L) < 70\),
we project them
 to \(H\) to get a line in \(\R^3/\Z^3\), 
 and hence an explicit triple \((p,q,r)\),
 and then we check Theorem~\ref{5lemma} explicitly
 for each of these triples.
 More explicitly, if we take the line \((uz,vz)\) on \(T\),
 then on \(H\) we explicitly get  the line \(L\)  on 
 \(H\) corresponding to
 \[ z \left( c u , c v , - b v  - a u  \right).\]
 For each of the fourteen \(H\), we need only consider the
lines on \(T\) whose height is at most \(70\).
There are \(5878\)  rational numbers with
height less than \(70\), and thus  taking into account
 \(0\) and \(\infty\) this
gives \(5880\) lines we need to check on each \(H\).
In fact there is some duplication, and after permuting the
orders and taking absolute values there are only
\(52258\) lines which we need to check in total.
If we also consider the lines of height less than \(84\),
the total number of such lines increases to \(76450\).
In these cases,
we check directly that there exists a \(t\) such that
\[ \left| \left( \frac{t}{p},\frac{t}{q},\frac{t}{r} \right) - \Lambda \right|
\ge 1 + \frac{5}{24} > 1 + \frac{1}{5}\]
in all cases we find such a \(t\), except
for the triples \((1,2,6)\), \((2,3,6)\), and \((3,4,12)\).
For these cases, the bound \(1 + 1/5\) is achieved
with \(t = 12/5\), \(t =6/5\), and \(t = 24/5\) respectively.
Note in all these computations we check
the condition against \(\Lambda\) rather than
\(\Z^3\).

This brings us to the hyperplane \(H\) corresponding to
\[ \frac{1}{p} + \frac{1}{q} + \frac{1}{r} = 0.\]
This hyperplane is different in that no points
are within anything less than \(1/2\)
of \((1/2,1/2,1/2)\).
We can and will assume in this section that \((p,q,r)\) have no common factor,
since Theorem~\ref{5lemma} is insensitive to scaling these parameters.
Moreover, by symmetry, we can assume that \((p,q,r)\) and positive and that
\[\frac{1}{p} = \frac{1}{q} + \frac{1}{r}.\]
Let us write \(q = p + s\), so
\[r  = \frac{p(p+s)}{s}.\]
For this to be an integer, \(s\) must divide \(p\), but then we find that 
\(p,q,r\) are all divisible by \(s\), and so the only possibility is that
\[ \left( \frac{1}{p},\frac{1}{q},\frac{1}{r} \right)
 = \left( \frac{1}{p}, \frac{1}{p+1},\frac{1}{p(p+1)}\right).\]
 We are free to assume that \(p > 0\).
 We now choose a suitable \(t\).
 Let
 \[t =  \frac{p(p+1)}{3} + 
 \begin{cases} 
 \displaystyle{\frac{p}{3}}, & p \equiv 0 \bmod 3, \\[10pt]
 0, & p \equiv 1 \bmod 3, \\[10pt]
\displaystyle{ - \frac{p+1}{3}}, & p \equiv 2 \bmod 3. \end{cases}
 \]
 and then
 \[ | t \v - \Lambda| = 1 + \frac{1}{3} - 
2 \begin{cases} 
 \displaystyle{\frac{1}{3(p+1)}}, & p \equiv 0 \bmod 3, \\[10pt]
 0, & p \equiv 1 \bmod 3, \\[10pt]
 \displaystyle{\frac{1}{3p}}, & p \equiv 2 \bmod 3. \end{cases}
 \]
 from which  Theorem~\ref{5lemma}
 (with the improved bound \(1 +5/24\) follows in these cases
 unless \(p = 2,3\).
 The remaining cases \(p=2\) and \(p=3\)
corresponding to \((p,q,r)\) being \((2,3,6)\) and \((3,4,12)\)
respectively which we have already verified.

\begin{center}
\begin{figure}[H]
  \includegraphics[width=6cm]{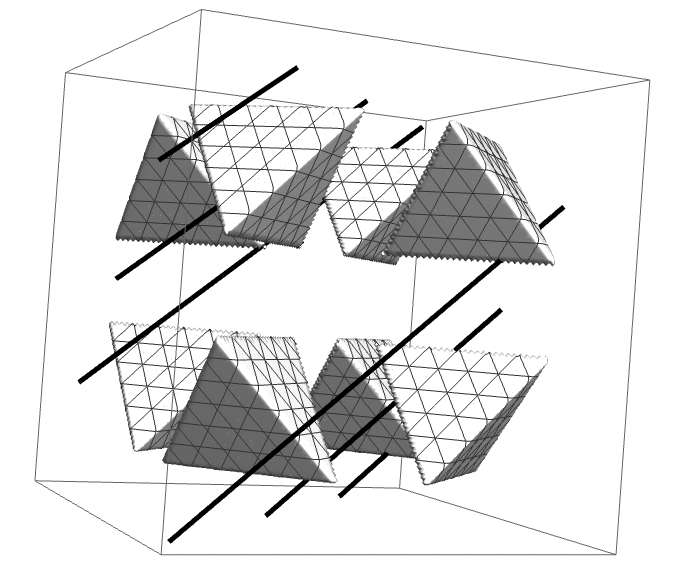}
  \caption{The  region \(|(x,y,z) - \Lambda| \ge 6/5\)
  in \(\R^3/(2 \Z)^3\) together with the line \((t/3,t/12,t/4)\).}
  \label{fig:3Db}
  \end{figure}
\end{center}

\subsection{The co-dimension two case} \label{sec:codim2}
 For the
\[ \binom{6337}{2} = 20075616\]
pairs of different lines, whenever they form a vector space of dimension two,
there is a unique line in \(\R^3\) which is orthogonal to both of them
given by the cross product. After normalizing these cross products
so they are either zero or normalized so that all terms
are non-negative 
and  the entries are coprime integers,
there are \(266743\) terms (these numbers  include the zero vector as well).
The entries with at least one zero do not correspond to any \((p,q,r)\).
If \((a,b,c)\) is any other triple, then we may take
\((p,q,r) = (bc,ac,ab)\) up to scalar.
For all such triples,
we  choose (up to) \(1000\) random
points \(t\) and stop if we have a value with
\[ \left| \left( \frac{t}{p}, \frac{t}{q}, \frac{t}{r} \right) - \Lambda \right| \ge 
1 + \frac{5}{24} > 1 + \frac{1}{5}.\]
This succeeds in every case except for \((p,q,r) = (1,2,6)\), \((2,3,6)\) and \((3,4,12)\).
This takes \(70\) seconds with one
core in \texttt{magma}.
 The  remaining cases also come up when
we consider the low height lines on the other exceptional hyperplanes.
Hence Theorem~\ref{5lemma}
is proved in the exceptional case.

\subsection{Relating bounds for 
\texorpdfstring{\(e(t)\)}{et} and 
\texorpdfstring{\(|t \v - \Lambda|\)}{tvminusLambda}}
We are now reduced to the case where we may
assume (see~\eqref{smaller}) that
\begin{equation}
\label{best}
E[(e(t)-24)^{12}] \ge 24^{12}
\end{equation}
In particular, there exists a \(t\) with \(e(t) \ge 48\).

We have the following:
\begin{lemma} \label{approx}
Let \(\varepsilon \in [0,1/6]\).
Let \(\v  = (x,y,z) \in \R^3/\Z^3\),and 
suppose that
\[ | \v -\Z^3 | \le 1 + 3 \varepsilon.\]
Then
\begin{equation}
\label{epsbound}
e(x,y,z) = (8 \sin(\pi x) \sin(\pi y) \sin(\pi z))^2  \le
64 \sin^6 \left(\frac{\pi}{3} + \pi \varepsilon \right).
\end{equation}
\end{lemma}

\begin{proof}
By symmetry, we may assume that \(0 \le x,y,z \le 1/2\).
Since \(\sin^2( \pi t)\) is increasing in this range, and
the function on the RHS 
of~\eqref{epsbound}
is also increasing  for \(\varepsilon\) in
this interval, we may assume that \(|tv - \Lambda| = 1 + 3 \varepsilon\)
and then maximize \(e(t)\) in the parameters  \((x,y,z)\).
The choice of \((x,y,z)\) means that \((0,0,0)\) is the closest lattice point, 
and so
\[x+y+z=1 + 3 \varepsilon.\]
Let us try to maximize
\[e(x,y,z) = 64 \sin^2(x \pi) \sin^2(y \pi) \sin^2(z \pi)\]
subject to the given constraints and show that \(x=y=z\).
We can use the method of Lagrange multipliers, that is, to consider
\[e(x,y,z) + \lambda (x+y+z - 1 - 3 \varepsilon),\]
and thus we find
\[ 
\begin{aligned}
\lambda + 128 \pi \cos(\pi x) \sin(\pi x) \sin(\pi y)^2 \sin(\pi z)^2 &  = 0, \\
\lambda + 128 \pi \cos(\pi y) \sin(\pi y) \sin(\pi x)^2 \sin(\pi z)^2 &  = 0, \\
\lambda + 128 \pi \cos(\pi z) \sin(\pi z) \sin(\pi x)^2 \sin(\pi y)^2 &  = 0, \\
x + y + z  & =  1 + 3 \varepsilon.
\end{aligned}
\]
The difference of the first two quantities is
\[- 128 \pi \sin(\pi x) \sin(\pi y) \sin(\pi(x-y)) \sin(\pi z)^2.\]
From this and is symmetrizations we deduce that any pair
of parameters in
\(x,y,z\) are either equal to each other, or they are equal to \(1/2\).
Considering the cases in turn we are led to the result.
\end{proof}

Returning to our argument, by~\eqref{best} we know
  there exists a point~\(t\) for which
\begin{equation}
\label{contra}
e(t) > \sqrt[12]{24^{12}} + 24  =  48,
\end{equation}
By Lemma~\ref{approx}, we deduce
that either  $|t \v - \Lambda| \ge | t \v - \Z^3|  \ge 1 + 1/5$,
or we must have
\[e(t) \le 64 \sin(2 \pi/5)^6
= (5 (5 + 2 \sqrt{5})) = 47.3606\ldots < 48.\]
But this is
 contradicts~\eqref{contra}, completing the proof
 of Theorem~\ref{5lemma}.
 
 Note that in the exceptional cases, we proved (with
 finitely many explicitly given exceptions) the lower bound
 \(1 + 5/24\). We observe that  
 the inequality \(e(t) \ge 48\) leads to the improved bound
\[\frac{3}{2} - \frac{
\displaystyle{3 \arccos \left( \frac{3^{1/6}}{2^{1/3}}\right)}
}{\pi} =  1.206646\ldots < 1.20833\ldots = 1 + \frac{5}{24},\]
which justifies Remark~\ref{improved}.
\end{proof}

\subsection{The geometry of numbers}
\label{sec:minkowski}

This section is not used in the paper, but is here as a complement
to the rest of \S~\ref{sec:fourier}, and gives a second approach to
proving Theorem~\ref{5lemma}. It ultimately reduces
Theorem~\ref{5lemma}
 to computations on hyperplanes of a sort already
 seen in \S~\ref{sec:fourier}.
 These ideas are inspired by the work of Jones, particularly~\cite{Jones1}.
We recall:

\begin{lemma} \label{5lemmaagain}
Let~$(p,q,r)$ be integers. Then there exists a real number $t$ such that
\[ \left| \left( \frac{t}{p}, \frac{t}{q}, \frac{t}{r} \right) - \Lambda \right| \ge 1 + \frac{1}{5}.\]
\end{lemma}

Note that this theorem is insensitive to a linear scaling
of the integers \((p,q,r)\).

Let \((p,q,r)\) be a triple of integers, not all exactly divisible by \(2\). Let
\[\v = (1/p,1/q,1/r) + (1/2,1/2,1/2),\]
and let \(\Lambda\) be the lattice \(\v + \Z^3\).
Let~\( | \cdot | = |\cdot|_1\).
The covolume of the lattice is \(n^{-1}\) where \(n\) is
 the smallest integer
so that \(n \v \in \Z\). Either at least one of \(p\), \(q\), and \(r\) is odd,
in which case \(n = [2,p,q,r]\),
or one of \(p\), \(q\), or \(r\) is divisible by \(4\),
in which case \(n = [p,q,r] = [2,p,q,r]\).
 In particular, \(n\) is always even.

Let~$\Phi^{\vee}$ denote the dual lattice, that is, the set of
vectors~$\w$ such that
$$\w \cdot \v \in  \Z$$
for all~$\v \in \Z$.
Note that~$\Phi^{\vee}$ contains~$n \Z^3$, but the covolume of~$\Phi^{\vee}$
is~$n$. (The precise covolumes will not be relevant for our computations,
however.)
Let \(\v_1\) denote a vector in \(\Phi\) of smallest norm,
and then \(\v_2\) the next smallest norm vector not in the space
generated by \(\v_1\), and then \(\v_3\) the final such vector.
These are norms with respect to \(F =  | \cdot |_1\).
Since \(n\) is even, at least one such vector must involve 
an odd multiple of \(\v\).
The dual lattice acquires a polar distance function \(F^*\) which
is \(| \cdot |_{\infty}\), the $\ell_{\infty}$-norm.
Correspondingly, let \(\w_i\) denote such a choice of vectors
in the dual lattice, and let the lengths of the \(\v_i\) and
the \(\w_i\) be \(\lambda_i\) and \(\mu_i\) respectively with
respect to the norms \(| \cdot |_1\) and \(| \cdot |_{\infty}\).

\begin{lemma}
Suppose that there do not exist integers \(a,b,c\), not all zero, with
\[\begin{aligned} \frac{a}{p} + \frac{b}{q} + \frac{c}{r}  & = 0,  \\
 \max(|a|,|b|,|c|) &  \le 20.
\end{aligned}
\]
Then Lemma~\ref{5lemmaagain} holds for \((p,q,r)\).
\end{lemma}

This is somewhat similar (if weaker) to what we proved using Fourier analysis
in \S~\ref{sec:fourier}, but the proof instead
uses the geometry of numbers.

\begin{proof}
Since Theorem~\ref{5lemma} is invariant up to scaling,
we may assume that the \((p,q,r)\) are not all exactly divisible by \(2\).
Note that if \(k\)  is an \emph{odd} integer and \(| k \v - \Z^3|\)
is small, then 
\[ \left(\frac{k}{p},\frac{k}{q}, \frac{k}{r} \right)\]
will be close to \((1/2,1/2,1/2) \bmod 1\). The problem is that small
vectors may not involve odd multiples of \(\v)\). In particular,
the smallest vector will typically be
 \[2 \v - (1,1,1) = (2/p,2/q,2/r).\]
On the other hand, the vectors \(\v_1,\v_2,\v_3\) cannot \emph{all}
have an even coefficient of \(\v_1\), and hence what we want to show
that that \(\v_3\) is not too large.

Returning to the statement,
The claim we want to establish  is that we can find a small linear relation. This would
mean that we can find an element in the dual lattice of very small length.
Hence 
we first assume this is impossible, and that the smallest vector \(\w_1\) has length
 \(\mu_1 \ge 20\). Note that the length is an integer since \(\Phi^{\vee} \subset \Z^3\).
 The first key step is to use Mahler's duality theorem~\cite[\S VIII.5, Thm VI]{CasselsBook},
 which says that that
$$1 \le \lambda_i \mu_{4-i} \le 3! = 6.$$
We deduce from this that
\[\lambda_i \le \lambda_3 \le \frac{6}{20}.\]
The coefficient of \(\v\) in \(\v_1,\v_2,\v_3\) cannot all be divisible by \(2\).
Thus there exists a vector of length bounded by \(6/20\) of the form
\(k \v + \Z^3\) with \(k\) odd. It follows that
\[ \left| \left(\frac{k}{p},\frac{k}{q},\frac{k}{r}\right)
+ \left(\frac{k}{2},\frac{k}{2},\frac{k}{2} \right) - \Z^3\right|
\le \frac{6}{20},\]
and thus, since \(k\) is odd,
\[ \left| \left(\frac{k}{p},\frac{k}{q},\frac{k}{r}\right)- \Z^3\right| \ge \frac{3}{2} - \frac{6}{20} = 
1 + \frac{1}{5}.\]

Thus we are done unless
\(\mu_1 < 20\), and thus there
exists 
a triple \((a,b,c)\) with
\begin{equation}
\label{nonzero}
\frac{a}{p} + \frac{b}{q} + \frac{c}{r} \equiv \frac{a+b+c}{2} \bmod 1.
\end{equation}
and 
\[ |a|,|b|,|c| < 20.\]
We are almost done, except the sum
on the RHS of \eqref{nonzero} might be non-zero.
Note, however,  \(k\) constructed above is an integer, but we need only
produce a rational number \(t = k/d\).
In particular, since we can always scale \((p,q,r)\)
by a large scalar, we can ensure that
\[ \frac{|a|}{p} + \frac{|b|}{q} + \frac{|c|}{r} < \frac{1}{2}.\]
(Remember that the coefficients in the  numerator are at most \(20\).)
Thus we are reduced to considering \((p,q,r)\) which lie
on the hyperplanes
\[ \frac{a}{p} + \frac{b}{q} + \frac{c}{r} = 0, \]
for integers \(a,b,c\) with
\[ 
 \max(|a|,|b|,|c|) \le 20.\]
 \end{proof}

From this point onwards, we are reduced to considering the exceptional
hyperplanes. This is exactly what we had arrived at in \S~\ref{sec:fourier},
although only after more work. On the other hand, that work also
resulted in fewer exceptional hyperplanes to consider. While one can certainly
envisage a proof of Lemma~\ref{5lemmaagain} following on in a similar way,
we are content with our original proof of (the same result) Theorem~\ref{5lemma}.

\section{The Jacobsthal Function}
\label{sec:jacob}

\begin{df} The primordial \(P_r = \prod_{k=1}^{r} p_k \) is the product of the first \(r\) prime numbers.
\end{df}

\begin{df} \label{defjac} Let~\(n\) be  a positive integer. Let~\((n,d) = 1\). The Jacobsthal function~\(J(n)\) is the smallest integer such that
any arithmetic progression:
\[a, a + d, \ldots, a + (J(n) - 1)d\]
of length~\(J(n)\) 
contains an element~\(a+kd\) which is coprime to~\(n\).
\end{df}

We have:
\begin{theorem}[Kanold] \label{kan} Suppose that~\(n\) has at most~\(m\) distinct prime divisors. Then~\(J(n) \le 2^m\).
\end{theorem}

From this we get:

\begin{lemma} \label{sqbound} If~\(n > 4\), we have an inequality
\[J(n) \le 2 \sqrt{n}.\]
If \(n\) has at least \(16\) prime factors, then
\[J(n)^3 \le \frac{n}{24300}.\]
\end{lemma}

\begin{proof}
 If~\(n\) has one prime divisor, then~\(J(n) = 2< 2 \sqrt{n}\) for~\(n > 1\).
If~\(n\) has two prime divisors, then~\(J(n) \le 4 < 2 \sqrt{n}\) for~\(n > 4\). 
So we may assume that~\(J(n)\) has at least three prime divisors.
Suppose it has~\(m \ge 3\) prime divisors. Then
\[n \ge 2 \cdot 3 \cdot 5 \cdot \ldots p_m,\]
where~\(p_m\) is the~\(m\)th prime. But then
\[n \ge 2 \cdot 3 \cdot 5 \cdot 5 \ldots 5 = 6 \cdot 5^{m-2}.\]
On the other hand, by the previous theorem, \(J(n) \le 2^{m}\). 
Now it suffices to check that
\[J(n) \le 2^{m} \le \sqrt{4^{m}} \le 2 \sqrt{ 6 \cdot 5^{m-2}} \le \sqrt{n},\]
where the third inequality holds for all~\(m\).

Now for for \(n\) with \(m \ge r\) prime factors, and with \(p_{r+1} \ge 2^{k}\),
we have 
\[n \ge P_m \ge P_r \cdot 2^{k(m-r)},\]
 and so
\[ J(n)^k \le 2^{mk}  = 
\frac{2^{rk} P_r 2^{k(m-r)}}{P_r}
\le  \frac{2^{rk} n}{P_r}
= n \cdot \frac{2^{kr}}{P_r}.\]
Now certainly \(p_{17}  = 59 \ge 8 = 2^3\), and so it suffices to note that
\[ \frac{P_{16}}{2^{48}} = 115779.94\ldots > 24300.\]
\end{proof}

We can (and will) improve this inequality, at least for numbers of moderate size.
More importantly, we have~\cite{MR2945166} the following bounds on \(J(n)\) when \(n\) has moderately few factors:

\begin{lemma}[\cite{MR2945166}]  \label{better} Suppose that \(n\) has \(r\) distinct prime factors when \(r \le 24\). Then \(J(n)\)
is bounded above by~$U(r)$ given by the following table:
\begin{center}
\begin{tabular}{*4c}
 \multicolumn{4}{c}{\emph{Upper bounds for \(J(n)\)}}\\
 \toprule
 $r$ & $U(r)$ &  $r$ & $U(r)$  \\
\midrule
 $1$ & $ 2$ & $ 13$ & $ 74$ \\
 $2$ & $ 4$ & $ 14$ & $ 90$ \\
 $3$ & $ 6$ & $ 15$ & $ 100$ \\
 $4$ & $ 10$ & $ 16$ & $ 106$ \\
 $5$ & $ 14$ & $ 17$ & $ 118$ \\
 $6$ & $ 22$ & $ 18$ & $ 132$ \\
 $7$ & $ 26$ & $ 19$ & $ 152$ \\
 $8$ & $ 34$ & $ 20$ & $ 174$ \\
 $9$ & $ 40$ & $ 21$ & $ 190$ \\
 $10$ & $ 46$ & $ 22$ & $ 200$ \\
 $11$ & $ 58$ & $ 23$ & $ 216$ \\
 $12$ & $ 66$ & $ 24$ & $ 236$ \\
\bottomrule
\end{tabular}
\end{center}
\end{lemma}

The main result of~\cite{MR2945166} is that it is possible that 
if \(n\) has \(r\) prime factors that \(J(n) > J(P_r)\). On the other hand,
is is trivial that if \(n\) has \(r\) prime factors then \(n \ge P_r\).
There are also bounds available for larger
ranges of \(r\),
including~\cite{MR2476571,MR3315513}, but these will 
ultimately not be needed (but whose existence were still
psychologically useful).

\subsection{A useful lemma}

We begin by formulating a general lemma.

\begin{lemma}\label{approximatem} Let~\(m_{\R}\) be a real number, and let~\(a\), \(d\), \(N\),  be positive integers.
If~\((d,N)=1\),
then there exists an integer~\(m\) such that:
\begin{enumerate}
\item \((a+ d m)\) is prime to~\(N\).
\item \(|m - m_{\R}| \le J(N)/2\).
\end{enumerate}
If~\(a=1\), the assumption that~\((d,N)=1\) is unnecessary. 
\end{lemma}

\begin{proof} Let~\(m_{\Z}\) denote the nearest integer to~\(m_{\R}\), so~\(m_{\R} = m_{\Z} + \varepsilon\)
and~\(|\varepsilon| < 1/2\).
If~\((d,N)=1\), then
by definition, any arithmetic progression~\((a + d i)\) of length~\(J(N)\) contains an element coprime to~\(N\).
If~\(d\) has a common factor with~\(N\), then let~\(N'\) denote the largest factor of \(N\) which is prime to \(d\),
so \(N/N'\) is only divisible by primes dividing \(d\). 
Then \((d,N') = 1\) and any arithmetic progression~\((a+ d i)\) of length~\(J(N') \le J(N)\) contains an element
coprime to \(N\). But the prime factors of \(N/N'\) divide \(d\), so if \(a=1\) then they are coprime to \(1+d i\),
and so every element in this sequence is coprime to \(N/N'\) and hence once an element is co-prime
to \(N'\) it is co-prime to \(N\).
We choose the sequence as follows:
\begin{enumerate}
\item If~\(J(N) = 2k+1\), we choose the sequence
\[a + d (m_{\Z} + i), \quad i = -k, \ldots, k.\]
We deduce there is a suitable~\(m\) with~\(|m_{\Z} - m| \le k\), and so
\[|m_{\R} - m| \le k + |\varepsilon| \le k+\frac{1}{2} = \frac{J(N)}{2}.\]
\item If~\(J(N)=2k\), we choose the sequence as follows:
\[\begin{aligned}
\text{If \(\varepsilon > 0\),} \quad a + d (m_{\Z} + i),  \quad i  & = -(k-1), \ldots, k,  \\
\text{If \(\varepsilon \le 0\),} \quad a + d (m_{\Z} + i),  \quad i  & = -k, \ldots, (k-1).
\end{aligned}
\]
If~\(m\) comes from an~\(i\) with~\(|i| \ne k\),
then~\(|m - m_{\Z}| \le k-1\) and~\(|m - m_{\R}| \le k - 1/2 = (J(N)-1)/2\).
If~\(|i| = k\), then, by construction, \(m > m_{\Z}\) if and only if~\(m_{\R} > m_{\Z}\),
and so
\[|m - m_{\R}| \le |m - m_{\Z}| = k = \frac{J(N)}{2}.\]
\end{enumerate}
\end{proof}

\section{Lower dimensional versions and Galois twists}
\label{sec:lower}

Our constructions are somewhat inductive, so it will be useful
to understand versions of this problem in dimensions one and
two.  Along the way we also introduce a new technique which
we call twisting. Some of the arguments in this section are also
ultimately going to be used in cumulative way, where a preliminary version of 
one lemma is used as input in a second argument which is then used
to strengthen the original lemma.

\subsection{A one dimensional version}
Let's consider the one dimensional version of this problem
with \(\Lambda = 2 \Z \subset \R\).
If we are given a \(n\), how accurately can we choose a  \(k\) with
\((k,2n)=1\) and with \(k/n \bmod 2\) in some given range?
Equidistribution implies that we can make it more and more accurate the
larger \(n\) is. Bounds on the Jacobsthal function allow us to prove
this for some explicit lower bound in \(n\),
and then checking directly for smaller \(n\) we can prove a result for all \(n\).

\begin{lemma}  \label{onedim}  
Let \(n > 1\).
There exists an integer \((k,2n) = 1\) so that if \(x = k/n \bmod 2\),
then \(x \in [1/6,1/2]\). 
If we  assume that \(n \ge 15\) and \(n \ne 18,21,33\), then
we can additionally find \(k\) so that
\(x \in [4/10,1/2]\). 
\end{lemma}

\begin{remark} With our notation (see Definition~\ref{notation}), we can also write the
first condition that \(x =k/n \bmod 2\) lies in \([1/6,1/2]\) as
\[ \left| \frac{k}{n} - \frac{1}{3} - 2 \Z \right| \le \frac{1}{6}.\]
\end{remark}

\begin{proof} Write \(k=1+2m\),   and  choose a real number \(m_{\R}\) 
so that \(k/n \bmod 2\)  lands in the middle
of this interval, which is \(x = 1/3\).  Clearly we can take
\[m_{\R} = \frac{n}{6} - \frac{1}{2}.\]
Now by Lemma~\ref{approximatem} we can find an integer \(m\) so
that \(k = 1+2m\) is prime to \(n\) and \(|m - m_{\R}| \le J(N)/2\),
where \(N  = n'\) is the largest odd factor of \(n\). (Certainly if \(m\) is
an integer then \(1+2m\) is automatically odd and prime to \(2\).)
We then find that 
\[ \frac{k}{n} = \frac{1+2m}{n} = \frac{1+2 m_{\R}}{n}
+ \frac{2 (m - m_{\R})}{n} = \frac{1}{3} + \frac{2 (m - m_{\R})}{n} \bmod 2.\]
We are done as long as
\[ \frac{ 2 |m - m_{\R}|}{n} \le \frac{1}{6},\]
but by construction we have
\[ \frac{ 2 |m - m_{\R}|}{n} \le \frac{J(n')}{n}.\]
Thus we are done as long as
\begin{equation}
\label{onedimeq}
J(n') \le \frac{n}{6}.
\end{equation}
Since \(J(n') \le J(n)\) and~\(J(n) \le 2 \sqrt{n}\) by Lemma~\ref{sqbound}, this is true for \(n \ge 144\). For smaller \(n\),
we find by computation that~(\ref{onedimeq}) still holds
unless 
\[n = 2, 3, 4, 5, 6, 7, 9, 10, 11, 15.\]
But for these values we can choose \(k\) so that 
\(k/n \bmod 2\) is as follows:
\[ \frac{1}{2}, \frac{1}{3}, \frac{1}{4},
\frac{1}{5}, \frac{1}{6}, \frac{2}{7},
\frac{2}{9}, \frac{3}{10}, \frac{3}{11}, \frac{4}{15} \in \left[ \frac{1}{6},\frac{1}{2} \right],\]
which completes the proof of the first claim.

Now suppose that \(n \ge 16\). The proof is very similar ---
we now take
\[m_{\R} = \frac{9n}{40} - \frac{1}{2},\]
and we
find an integer \(m\) so
that \(k = 1+2m\) is prime to \(n\) and \(|m - m_{\R}| \le J(N)\),
where \(N  = n'\) is the largest odd factor of \(n\). Then we have
\[ \frac{k}{n} = \frac{9}{20} + \frac{2 (m - m_{\R})}{n} \bmod 2.\]
We are done as above as long as
\begin{equation}
\label{anagain}
J(n') \le \frac{n}{20}.
\end{equation}
Since \(J(n') \le J(n)\) and~\(J(n) \le 2 \sqrt{n}\) by Lemma~\ref{sqbound}, this is true for \(n \ge 1600\). For smaller \(n \ge 15\) with \(n \ne 18,21,33\),
we find by computation that~(\ref{anagain}) still holds
unless 
\[n =
15, 16, 17, 19, 20,
22, 23, 24, 25, 26, 27, 28, 29, 30, 31,  34, 35, 36, 37, 38, 39, 42, 45, 51, 55, 57.\]
But for these values we can choose  the corresponding \(k\) as follows:
\[7, 7, 7, 9, 9, 9, 11, 11, 11, 11, 13, 13, 13, 13, 15, 15, 17, 17, 17, 17, 19, 19, 19, 25, 27, 25,\]
and then \(k/n\) always lies 
in the desired range, completing the proof.
\end{proof}

Naturally,
if required, one can always prove versions of this lemma allowing more
explicit exceptions.

\subsection{Twisting} \label{sec:twisting}
As mentioned in the introduction, one can approach
Conjecture~\ref{conjA} by thinking in terms of roots of unity
or in terms of lattice points. From the point of view of roots of
unity \(\zeta,\xi,\theta\), if the extension \([\Q(\zeta,\xi,\theta):\Q(\xi,\theta)]\)
is large, then \(\Gal(\Q(\zeta,\xi,\theta)/\Q(\xi,\theta))\) will move
\(\zeta\) around to within a small error of every point in the circle
while keeping the other roots of unity fixed. In this subsection
we describe the corresponding analogue for  the lattice point
version of the problem.

\begin{lemma} \label{sturdy} Let 
\(\displaystyle{m =  \frac{[p_1,p_2, \ldots, p_r,q]}{[p_1,p_2, \ldots, p_{r}]}}\).
Let \(m'\) be the largest odd  factor of \(m\).
Let \(k\) be  prime to 
\(\displaystyle{2 \prod_{i=1}^{r} p_i}\).
Then there exists another integer \(k'\) with the following properties:
 \begin{enumerate}
 \item \(k'\) is prime to \(\displaystyle{2 q \prod_{i=1}^{r} p_i}\).
 \item For \(i = 1, \ldots, r\), we  have
 \(\displaystyle{ \frac{k'}{p_i} \equiv \frac{k}{p_i} \bmod 2}\)
 \vspace{5pt} 
 \item We have
 \(\displaystyle{\left| \frac{k'}{q} - x  - 2 \Z \right| \le \frac{J(m')}{m}}\).
 \item \(k'\) is  prime to any auxiliary integer.
 \end{enumerate}
 \end{lemma}
 
 \begin{proof}
 Since adding \(\displaystyle{2 \prod_{i=1}^{r} p_i}\) to \(k'\)
 doesn't change any of the first three properties,
 we can easily ensure the fourth by the Chinese Remainder
 Theorem, so we concentrate on the first three conditions.
 
 We shall consider integers \(k'\) of the form
 \[k' = k + 2 [p_1, \ldots, p_m] i j \]
 for integers \(i,j\). Here \(i\) will vary and \(j\) is fixed, to be chosen later.
 Certainly \(k'\) is prime to \(2\) and the integers \(p_i\).
 Thus \(k'\)
 is prime to \(q\) if and only if it is prime to \(m'\).
 Finally, \(k'/p_i \equiv k/p_i \bmod 2\).
Now we need to make a judicious choice of the integer \(i\).
 There exists an~\(i_{\R}\) and~\(k_{\R}\)
 so that if
 \(k_{\R} = k + 2[p_1,\ldots, p_r]  i_{\R} j\),
  then
  \[k_{\R}/q \equiv x \bmod 2.\]
  By Lemma~\ref{approximatem},
 we can now find an integer~\(i'\) with~\(|i' - i_{\R}| \le J(m')/2\) so that~\(k'\)
 is prime to~\(m'\). Now let us make a careful  choice of~\(j\).
 We note that if we change~\(i\) to~\(i+1\) then~\(k'/q\) changes by
 \[2 \cdot \frac{[p_1,\ldots, p_r] j}{q} \bmod 2.\]
 By considering
the powers of any prime dividing either~\(p\), \(q\), or~\(r\), we find that in reduced terms this is equal to
 \[2 \cdot \frac{c j}{m} \bmod 2,\]
 where~\((c,m)=1\). 
We now choose a~\((j,m)=1\) so that~\(cj \equiv 1 \bmod m\),
which is possible because~\((c,m)=1\).
But then we have
  \[2 \cdot \frac{[p_1,\ldots,p_r] j}{q} \equiv  \frac{2}{m} \bmod 2.\]
Now by construction, we find that, modulo~\(2\), we have

 \[ \left| \frac{k'}{q} - x - 2 \Z \right|  = \left| \frac{k'}{q} - \frac{k_{\R}}{q} \right| 
 =\frac{2 |i' - i_{\R}|}{m} \le \frac{J(m')}{m}\]
   which gives the third condition.
  \end{proof}

\subsection{A two dimensional problem}

With the one dimensional version and with twisting in hand, we
now consider a two dimensional version of our problem.

\begin{lemma} \label{previous}
Let~\((p,q)\) be any pair of integers. Then there exists an integer~\((k,2pq)=1\) so that if~\(x = k/p \bmod 2\) and~\(y = k/q \bmod 2\),
then, up to reordering,
\[  \frac{1}{2} \ge x  \ge \frac{1}{6}, \quad  \frac{5}{6} \ge y \ge \frac{1}{6}.\]
\end{lemma}

\begin{proof} Write \((p,q) = (ad,bd)\) with \((a,b)=1\), and \(b > a\).
Let \(m = [p,q]/p =  b\).
By Lemma~\ref{onedim}, we may find a \(k\) with \(k/p \in [1/6,1/2] \bmod 2\). 
We now want to modify \(k\) so that \(k/p \bmod 2\) remains unchanged but \(k/q  \bmod 2\) lands in the desired interval.
By Lemma~\ref{sturdy},  we can keep \(k/p \bmod 2\) fixed, and make
 \[\displaystyle{\left| \frac{k}{q} - \frac{1}{2}  - 2 \Z \right| \le \frac{J(m')}{m}}.\]
Thus we are done
as long as
\[ \frac{J(m')}{m} \le \min \left( \left| \frac{1}{2} - \frac{1}{6}\right|, \left|\frac{1}{2} - \frac{5}{6}\right|\right) = 
\frac{1}{3}.\]
With \(m = b\), this only occurs for \(b = 1,2,3,5\). So it remains
to consider these remaining cases.

We claim that for all of the finitely many remaining pairs \((a,b)\), there
exists a real number \(k_{\R}\) so that --- after possibly reordering \(a\) and \(b\) --- 
the point
\[ \left( \frac{k_{\R}}{ad}, \frac{k_{\R}}{bd}\right)\]
lies on the vertical 
line segment~\(\Phi\) from~\([1/4,1/4]\) to \([1/4,3/4]\).
To prove this, we simply compute each of the seven lines
for a suitable ordering of \((a,b)\).
More concretely,
the lines of slopes 
\[1/1,2/1,3/1,3/2,1/5,2/5,3/5\]
 intersect~\(\Phi\)
  at the points with \(y\) coordinates
   \[1/4,2/4,3/4,3/8,9/20,11/20\]
    respectively.
By Lemma~\ref{approximatem},
we may find an integer \(k\) with \((k,2abd)=1\) and such that \(|k- k_{\R}| \le  J(2abd)/2\).
We now claim that the point
 \[ \left( \frac{k}{ad}, \frac{k}{bd}\right)\]
 lies in the correct interval. If \(k = k_{\R}\), then one value
 would be \(1/4 \bmod 2\) and the other in \([1/4,3/4]\).
 After changing to \(k\), the terms will deviate by at most
\(J(2abd)/(2ad)\) for the first term and \(J(2abd)/(2 b d)\) for the second.
Note that \(a \le b\) so these are both at most \(J(2abd)/(2 a d)\).
In order to stay inside the region, which for
one term is in \([1/6,1/2]\) and the other is  in  \([1/6,5/6]\),
 the terms can both vary as much as the distance from \(1/4\)
 to the closest boundary point \(1/6\), or the distance
 from a point in \([1/4,3/4]\) to the nearest boundary point 
 in \([1/6,5/6]\), or in other words by
\[\min\left(\left|\frac{1}{4} - \frac{1}{6}\right|, \left|\frac{3}{4} - \frac{5}{6} \right| \right)= \frac{1}{12}.\]
Thus we are done as long as
 \[  \frac{J(2 a b d)}{2 a d}
 \le \frac{1}{12},\]
 but since \(J(2 a b d) \le 2 \sqrt{2 a b d}\), this holds if
 \[ 2 \sqrt{2 a b d} \le \frac{ a d}{6},\]
 or
 \[d \le \frac{288 b}{a} \le 288 \cdot 5 =1440.\]
 There are \(7\) pairs \((a,b)=1\) with \(b \in \{1,2,3,5\}\) so this leaves
 \(4 \times 1440 = 7200\) cases to check directly.
 \end{proof}

In Lemma~\ref{previous},  one can do no
 better than
 equality in the case of~\((p,q) = (6,6)\). But we can give an improvement
 if we impose further lower bounds on \(p\) and \(q\).
 
 \begin{lemma} \label{previous2}Let~\((p,q)\) be any pair of integers with
 \(p,q \ge 15\) and \(p,q \ne 18,21,33\).   Write \((p,q) = (ad,bd)\) with \((a,b)=1\).
Then there exists an integer~\((k,2pq)=1\) so that if~\(x = k/p \bmod 1\) and~\(y = k/q \bmod 1\),
then
\[  \frac{5}{7} \ge x  \ge \frac{2}{7}, \quad  \frac{5}{7} \ge y \ge \frac{2}{7},\]
unless \((p,q) = (15,30)\).
 \end{lemma}
 
 \begin{proof}
 Write \((p,q) = (ad,bd)\) with \((a,b)=1\), and, without loss
 of generality, assume that \(b > a\).
Let \(m = [p,q]/p =  b\).
 By Lemma~\ref{onedim},
 we choose a \((k,2pq) = 1\) so that \(k/p \bmod 2\)
 lies in \([4/10,1/2]\).
 Arguing as in the proof of Lemma~\ref{previous},
 and using Lemma~\ref{sturdy},  we can keep \(k/p \bmod 2\) fixed, and make
 \[\displaystyle{\left| \frac{k}{q} - \frac{1}{2}  - 2 \Z \right| \le \frac{J(m')}{m}}.\]
Thus we are done
as long as
\[ \frac{J(m')}{m} \le \min \left( \left| \frac{1}{2} - \frac{2}{7}\right|, \left|\frac{1}{2} - \frac{5}{7}\right|\right) = 
\frac{3}{14}.\]
With \(m = b\), this only occurs for 
\begin{equation}
\label{blist}
b = 1, 2, 3, 4, 5, 6, 7, 9
\end{equation}
For \(a < b\) with \((a,b)=1\) and \(b\) in this list, but not of the exceptional
form above, we claim that there exists a \(k_{\R}\) such that
\[
\max \left(\left| \frac{k_{\R}}{ad} - \frac{1}{2} + \Z \right|,
\left| \frac{k_{\R}}{bd} - \frac{1}{2} + \Z \right| \right)
 \le 
\begin{cases} 
\displaystyle{\frac{1}{6}}, & (a,b) = (1,2), \\[10pt]
\displaystyle{\frac{1}{10}}, & (a,b) = (1,4), (2,3), \\[10pt]
\displaystyle{\frac{1}{14}}, & \text{otherwise}.
\end{cases} 
\]

We prove this but taking points \((t/a,t/b)\) with \(t \in \Z/1260\) and finding
the largest such point; for this method, the case \(1/14\) is optimal
for the choices \((a,b) = (1,6), (2,5), (3,4)\).
The next step is to now find an integer \(k\) with \(|k - k_{\R}| \le J(2abd)/2\)
and \((k,2abd) = 1\), and then, since \(1/ad \ge 1/bd\),
we need
\begin{equation}
\label{firsttime}
\frac{J(2abd)}{2ad} \le
\begin{cases}
\displaystyle{\frac{3}{14} - \frac{1}{6}  = \frac{1}{21}}, & (a,b) = (1,2), \\[10pt]
\displaystyle{\frac{3}{14} - \frac{1}{10} = \frac{4}{35}}, & (a,b) = (1,4), (2,3), \\[10pt]
\displaystyle{\frac{3}{14} - \frac{1}{14} = \frac{2}{7}}, & \text{otherwise}.
\end{cases}
\end{equation}
Using that \(J(2 a b d) \le 2 \sqrt{2 a b d}\),  we are 
done if
\[\frac{J(2abd)}{2ad} \le \frac{1}{C} \Rightarrow 
 a b d \le 2 b^2 C^2,\]
and hence we are done if
\begin{equation}
a b d \le
\begin{cases}
\displaystyle{882 = 2 \cdot 2^2 \cdot 21^2} & (a,b) = (1,2), \\[10pt]
\displaystyle{2450   = 2 \cdot 4^2 \cdot \left(\frac{35}{4}\right)^2} & (a,b) = (1,4), (2,3), \\[10pt]
\displaystyle{1985 > 2 \cdot 9^2 \cdot \left(\frac{7}{2}\right)^2} & \text{otherwise}.
\end{cases}
\end{equation}
and we can easily check these cases and find that the only
exception in this range is \((p,q) = (15,30)\).
 \end{proof}

 \subsection{Bounds for
 \texorpdfstring{\([p,q,r]/[p,q]\)}{[p,q,r]/[p,q]}}
 
 We can draw a few useful consequences out of the lemmas in
 this section. For example:

  \begin{lemma} \label{mbound}
 Let~\((p,q,r)\) be a triple, and let~\(m = [p,q,r]/[p,q]\).
 Then if~\(m\) is not in the following list:
 \begin{equation}
 \label{mlist}
 \begin{aligned}
 & 1, 2, 3, 4, 5, 6, 7, 9, 10, 11, 15 \end{aligned}
 \end{equation}
 then Theorem~\ref{maintheorem} is true for \((p,q,r)\).
 \end{lemma}
 
 \begin{proof} We first find~\((k,2pq)\) with~\(k/p \equiv x \bmod 2\) and~\(k/q \equiv y \bmod 2\)
 as in Lemma~\ref{previous}, so with
\[  \frac{1}{2} \ge x  \ge \frac{1}{6}, \quad  \frac{5}{6} \ge y \ge \frac{1}{6}.\]

By Lemma~\ref{computingM}, 
there exists a \(z\) such that \(|(x,y,z) - \Lambda| \ge 1 + 1/6\).
By Lemma~\ref{sturdy}, we find a~\(k'\) with \((k',2pqr) = 1\) and
 \[\frac{J(m')}{m} \ge e = |k'/r - z - 2 \Z|.\]
 By the triangle inequality, we have
 \[\left|\left(\frac{k'}{p},\frac{k'}{q},\frac{k'}{r}\right) - \Lambda \right|
 \ge 1 + \frac{1}{6} -  \frac{J(m')}{m},\]
 This means we are done  as long as
 \[\frac{J(m')}{m} \le \frac{1}{6}\]
 Since
 \[ \frac{J(m')}{m} \le \frac{J(m)}{m} \le \frac{2 \sqrt{m}}{m} \le \frac{1}{6}\]
this is automatic as soon as~\(m \ge 12^2 = 144\),
and then by computation  for most smaller \(m\) as well,  and we find the only exceptions lie in~(\ref{mlist}).
    \end{proof}
    
    There is also the following very similar variant:
    
     \begin{lemma} \label{mbound4}
 Let~\((p,q,r)\) be a triple with \((p,q)=(ad,bd)\) and \((a,b)=1\), and let~\(m = [p,q,r]/[p,q]\).
Suppose that
 \(p,q \ge 15\) and \(p,q \ne 18,21,33\).
If~\(m \ne 1,2,3,5,6\), then Theorem~\ref{maintheorem} is true for \((p,q,r)\).
 \end{lemma}
 
 \begin{proof}  By
 Lemma~\ref{previous2},
 we find \(k\) with \(k/p,k/q\) in \([2/7,5/7]^2\).
 Then, as in the proof of Lemma~\ref{mbound},
 we are done as long as
 and then we are in good shape as long as
 \[1 + \frac{2}{7}  -  \frac{J(m')}{m} \ge 1.\]
 But for \(m \notin \{1,2,3,5,6\}\) we have \(J(m')/m \le  2/7\).
 This leaves the exceptional pair~\((p,q) = (15,30)\),
 But we know that the possible \(r\)s must satisfy \([r,p,q]/[p,q] \le 15\),
 by Lemma~\ref{mbound}. This bounds \(r\) in all cases by \([p,q] \cdot 15
 \le 30 \cdot 15 =450\), and 
 and indeed the value of \(n = [2,p,q,r]\) is also 
 bounded by this quantity. But we can compute these cases directly.
\end{proof}

\section{The Regime where 
 \texorpdfstring{\(\min(p,q,r)\)}{minpqr} is large relative to
 \texorpdfstring{\(n\)}{n}}
\label{sec:nbig}

In this section, we study the problem where the \(p,q,r\) are not
too large compared to \(n\). The main theorem of this section is as follows:

   \begin{lemma} \label{circular}
   Suppose that \(\min(p,q,r) \ge m\). Let \(n = [2,p,q,r]\)  and
   suppose that  the number of distinct prime factors of \(n\) is \(r\).
   Then Theorem~\ref{maintheorem}
   holds for \((p,q,r)\) as long as
\[J(n) \le \frac{2m}{15}.\]
Moreover, this holds if \(r < A(m)\) for the values of \(A(m)\)
   in the table below,  or if  \(n < B(m)\).
  \begin{center}
\begin{tabular}{*3c}
 \multicolumn{3}{c}{\emph{Lower bounds for \(n\) in terms of
 \(min(p,q,r)\)}}\\
 \toprule
 $m = \min(p,q,r)$  & $A(m)$ & $B(m)$  \\
\midrule
 \(105\) & $6$ &  \( 30030\)  \\
 \(165\) &   $7$ & \( 510510\)  \\
 \(195\) & $8$ & \( 9699690\)  \\
 \(255\) &   $9$ & \( 223092870\)  \\
 \(300\) &  $10$ & \( 6469693230\)  \\
 \(345\) &  $11$ &  \( 200560490130\)  \\
 \(435\) &  $12$ & \( 7420738134810\)  \\
 \(495\) & $13$ &  \( 304250263527210\)  \\
 \(555\) & $14$ &  \( 13082761331670030\)  \\
 \(675\) & $15$ &  \( 614889782588491410\)  \\
 \(750\) & $16$ &  \( 32589158477190044730\)  \\
 \(795\) &  $17$ & \( 1922760350154212639070\)  \\
 \(885\) & $18$ &  \( 117288381359406970983270\)  \\
 \(990\) &  $19$ & \( 7858321551080267055879090\)  \\
 \(1140\) &  $20$ & \( 557940830126698960967415390\)  \\
\bottomrule
\end{tabular}
\end{center}
 \end{lemma}
  
  \begin{proof}
  The basic idea 
is to start with Theorem~\ref{5lemma}, which guarantees a real
number \(t = t_{\R}\) so that
\[ \left| \left( \frac{t}{p}, \frac{t}{q}, \frac{t}{r} \right) - \Lambda \right| \ge 1 + \frac{1}{5}.\]
By Lemma~\ref{approximatem}, there
exists a \(k \in \Z\) with \((k,n)=1\) and \( |k - t| \le J(n)/2\). We then have
\[ \left| \left( \frac{k}{p}, \frac{k}{q}, \frac{k}{r} \right) - \Lambda \right|
\ge 1 + \frac{1}{5} - \frac{J(n)}{2} \left( \frac{1}{p} + \frac{1}{q} + \frac{1}{r} \right).\]
Thus Theorem~\ref{maintheorem} holds for \((p,q,r)\) as long
as
 \[\frac{J(n)}{2} 
  \left( \frac{1}{p} + \frac{1}{q} + \frac{1}{r} \right)
   \le \frac{1}{5}\]
   If \(\min(p,q,r) \ge m\), then we are done as long as
\[J(n) \le \frac{2m}{15}, \ \text{or} \ m \ge \frac{15 J(n)}{2}.\]
Suppose this inequality fails.
Then Lemma~\ref{better}, for the various \(m\),
this implies that \(n\) has at least \(A(m) = k\) prime factors
for various \(k\) and thus
\[n \ge B(m) = P_k = \prod_{i=1}^{k} p_k.\]
\end{proof}

In light of this, 
it will be useful to understand what
happens when one of \(p\),  \(q\), and \(r\) is small. We develop
some tools now to understand this case.

\section{The Regime where \texorpdfstring{\(\min(p,q,r)\)}{minpqr} is fixed}
  \label{sec:small}
 In \S~\ref{sec:nbig}, we obtained some control when \(n\) was not too
 large compared to \(\min(p,q,r)\). In this section, we are now in a position to rule out one form of counterexamples
to Theorem~\ref{maintheorem} when \(\min(p,q,r)\) is small.
 Supposing that \(p\) is small, we typically write \((q,r) = (ad,bd)\)
with \((a,b)=1\). Our first observation is that, for a fixed \(p\),
we need only consider finitely many pairs \((a,b)\):

\begin{lemma} \label{bound}
Consider triples~\((p,ad,bd)\), where~\((a,b)=1\) and \(b \ge a\).
Then if this is a  counterexample
to Theorem~\ref{maintheorem}, we have the following inequalities:
\[\begin{aligned}
b & \le 15p, \\
a & \le 15p, \\
ab & \le 165p.
\end{aligned}
\]
\end{lemma}

\begin{proof} Let us compute~\(m = [p,q,r]/[p,q]\)
 and \(m' = [p,q,r]/[p,r]\). 
Certainly~\([p,q,r]\) is divisible by~\(bd\). The largest factor of~\(bd\) in~\(ad\) is~\(d\),
and the largest factor of~\(p\) in~\(bd\) divides~\(p\), so
\[\left. \frac{b}{(b,p)} \right| m, \left. \qquad \frac{a}{(a,p)} \right| m' \]
By Lemma~\ref{mbound}, we have 
\[m,m' \in  1, 2, 3, 4, 5, 6, 7, 9, 10, 11, 15.\]
The inequality for \(a\) and \(b\) then follows.
On the other hand, we see that \((m,m')=1\) since a prime dividing \(m\)
or \(m'\)
divides \(r\) or \(q\) respectively to a higher power than any other element.
So from the divisibility
\[\left.  \frac{ab}{(ab,p)} = \frac{b}{(b,p)}\frac{a}{(a,p)} \right| mm'\]
we get \(ab | pmm' \le 165p\).
\end{proof}

\subsection{The strategy for small
\texorpdfstring{\(p\)}{p}}
Now let us explain our approach to triples of the form \((p,ad,bd)\).
We first choose a \(k\) so that \(x = k/p \bmod 2\) lies in \([1/6,1/2]\).
We then consider new \(k'\) of the form \(k' = k + 2pm\),
fixing the congruence class modulo \(2p\). The idea is that this is not too
restrictive for \(p\) small, and guarantees that \(k/p \bmod 2\) is not too small.
In order to choose \(m\), we find a real number \(m_{\R}\) so that, with
\[y = \frac{k + 2pm_{\R}}{ad} \bmod 2, \quad
z = \frac{k + 2pm_{\R}}{bd} \bmod 2,\]
we have
\[ |(x,y,z) - \Lambda| \ge \frac{7}{6}.\]
Note that the existence of such a real number \(m_{\R}\) depends only
on \(x\) and \(a\) and \(b\) and not on \(d\).
Then we show, for \(d\) sufficiently large, we can find an integer \(m\)
sufficiently close to \(m_{\R}\) so that, for \(k' = k+2pm\), 
we have \((k',2pad) = 1\) and
the triple \(\v = (k'/p,k'/q,k'/r)\) is close enough to \((x,y,z)\) to ensure
that \(|\v - \Lambda| \ge 1\).

The following is elementary:
\begin{lemma} Let \(x \in [1/6,1/2]\).
Let~\(\Omega_{[0,1]}(x) \subset [0,1]^2\) denote the region of points~\((y,z)\) in the plane
satisfying the following four inequalities:
\[ \begin{aligned}
1   - \frac{1}{6} + x & \ge z + y   \ge 1  + \frac{1}{6} - x, \\
1 - \frac{1}{6} - x & \ge z - y  \ge -1 + \frac{1}{6} + x, \end{aligned}
\]
Let \(\Omega(x) \subset [0,2]^2\) denote the union of four copies
of \(\Omega_{[0,1]}(x)\) rotated by multiples of~\(\pi/2\) around
the point \((1,1)\).
Then \(|(x,y,z) - \Lambda| \ge 1 + 1/6\) for \((y,z) \in \Omega(x)\).
\end{lemma}
Note that when~\(x=1/6\) then~\(\Omega(x)\) consists of four lines,
and when \(x = 1/2\), it consists of four squares, and otherwise
it consists of four (non-square) rectangles. Here are pictures
of \(\Omega(x)\) for \(x =1/6\), \(1/3\) and \(1/2\) together
with 
\[ \Omega:=\cap_{x \in [1/6,1/2]} \Omega(x),\]
where \(\Omega\) is the region given by four rotated
copies of the line:
\[x + y = 1, 1/3 \le x,y \le 2/3.\]
Pictures of these regions are as follows:
\begin{center} 
\begin{figure}[H]
  \includegraphics[width=9cm]{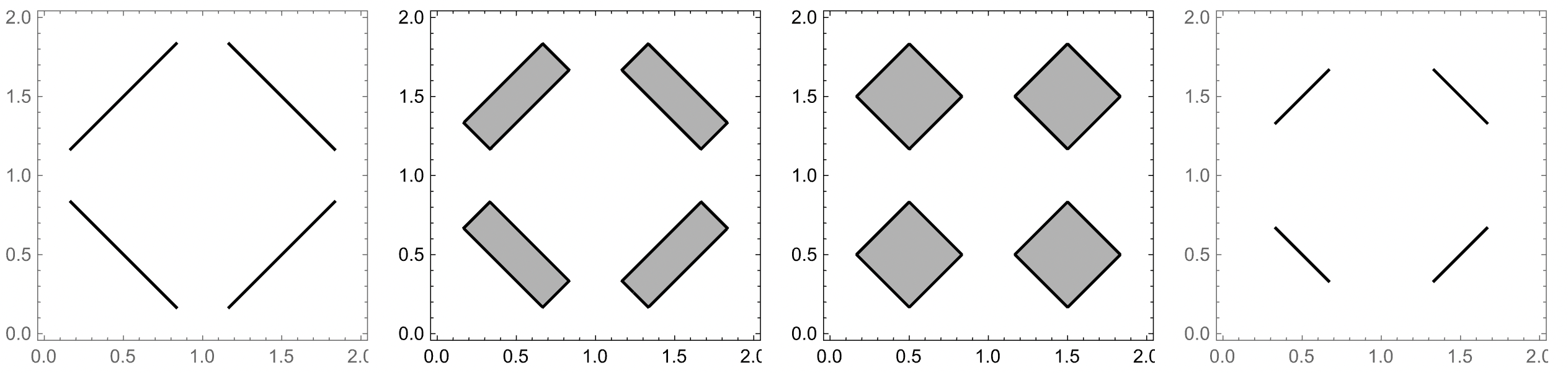}
  \caption{The region \(\Omega(x)\) for \(x = 1/6, 1/3, 1/2\) and the intersection \(\Omega\) of all \(\Omega(x)\).}
  \label{fig:omega}
  \end{figure}
\end{center}

\begin{lemma}  \label{geometry}
Let \((a,b)\) be any pair of coprime positive integers.
Then there exists a real \(t\) so that \((t/a,t/b) \bmod 2\) lies on \(\Omega\).
\end{lemma}

\begin{proof} Let \(r\) denote the slope, and consider equivalently
the line \((t,t r) \bmod 2\).
Without loss of generality by symmetry, we may assume that \(0 < r \le 1\).
In particular, we may assume that
\[r \in \left[1,\frac{1}{2} \right] 
\cup  \left[\frac{1}{2},\frac{4}{11} \right] 
\cup  \left[\frac{4}{11},\frac{1}{4} \right] 
\cup  \left[\frac{1}{8},\frac{1}{4} \right] 
\cup \bigcup_{n \ge 1} 
 \left[\frac{1}{4 + 6n},\frac{2}{5+6n} \right] 
\]
To see this contains \([0,1)\), note that
\[\begin{aligned}
\frac{2}{11} &  > \frac{1}{8},\\
\frac{2}{5+6n} & > \frac{1}{4 + 6(n+1)}, \ n \ge 1 \ge \frac{1}{2}.
\end{aligned}\]
Now for any \(r\) in the final segments the line intersects the line between
the two points
\[ \left( 2n + 1 + \frac{1}{3}, \frac{1}{3} \right), 
\left( 2n + 1 + \frac{2}{3}, \frac{2}{3} \right).\]
For the initial four segments, they also intersect four corresponding
lines as observed in Figure~\ref{fig:lines}.
\begin{center}
\begin{figure}[H]
  \includegraphics[width=6cm]{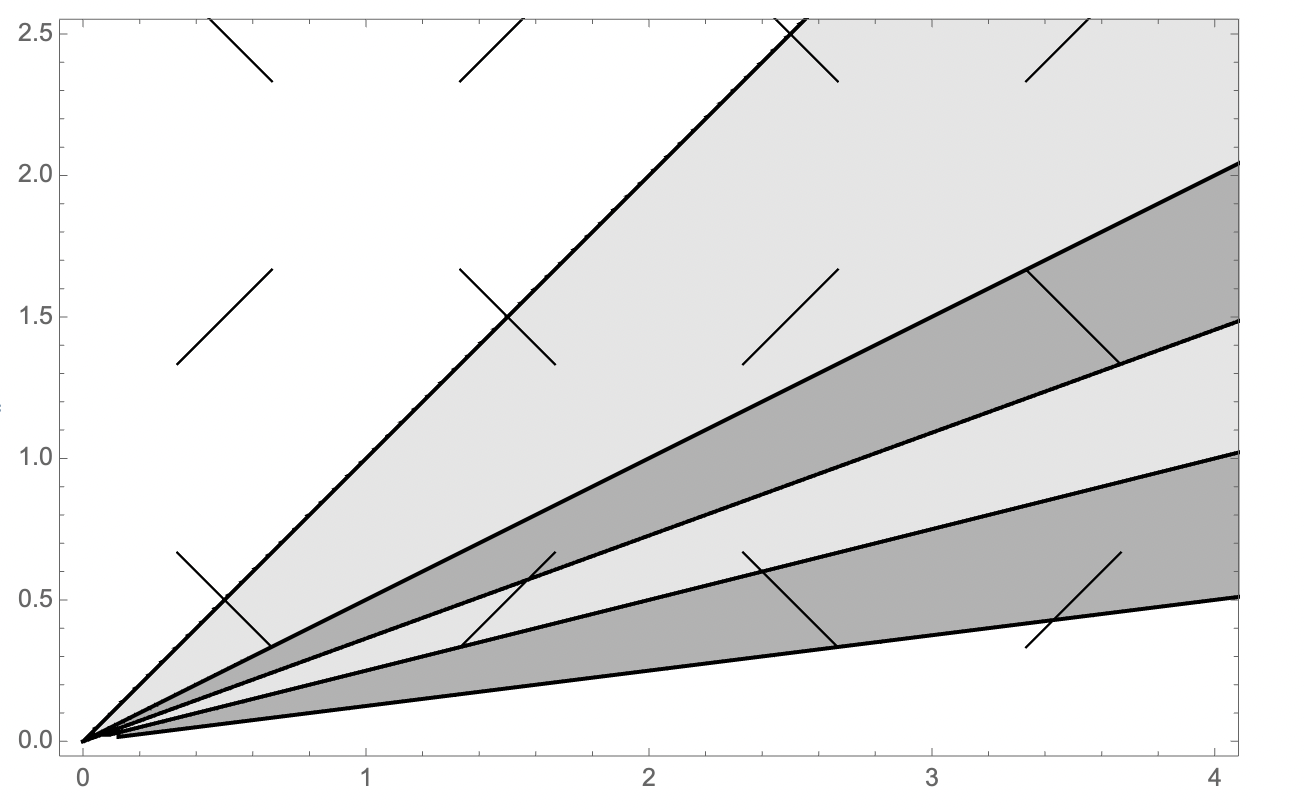}
  \caption{Lines intersecting \(\Omega\).}
  \label{fig:lines}
  \end{figure}
\end{center}
\end{proof}
Now let us return to the~\((p,an,bn)\) problem.  

\begin{lemma} \label{medium} Suppose that \(p \le 33\).
Then Theorem~\ref{maintheorem}
holds for \((p,q,r) = (p,an,bn)\).
\end{lemma}

\begin{proof}  Write \((p,q,r) = (p,ad,bd)\)
with \((a,b)=1\) and \(b \ge a \ge 1\).
\ By Lemma~\ref{bound},
we may assume that \(a,b \le 15 p\),
and \(a b \le  165p\).
If \(n = [2,p,q,r]\), then \(n\) divides \(2abdp\).
By Lemma~\ref{onedim}, we choose a \((k,2p) = 1\)
with \(k/p \bmod 2 \in [1/6,1/2]\).
By Lemma~\ref{geometry},
there exists real numbers~\(m_{\R}\) and \(t_{\R}\) so that
\(k + 2p m_{\R} = t_{\R}\) and
\[ \left( \frac{t_{\R}}{da},  \frac{t_{\R}}{db} \right) \in \Omega \bmod 2.\]
By Lemma~\ref{approximatem},
there exists an~\(m \in \Z\) so that~\((k + 2pm,abd) = 1\)
with~\(|m - m_{\R}| \le J(abd)/2\).
We find that, modulo~\(2\),
\[\begin{aligned}
& \left( \frac{k + 2pm}{p}, \frac{k + 2pm}{ad},  \frac{k + 2pm}{bd} \right) \\
 & = \left(\frac{k}{p},\frac{k + 2p m_{\R}}{ad},\frac{k + 2p m_{\R}}{bd} \right)
+ 2p (m-m_{\R}) \left(0,\frac{1}{ad},\frac{1}{bd}\right). \\
& = \left(\frac{k}{p},\frac{t_{\R}}{ad},\frac{t_{\R}}{bd} \right)
+ 2p (m-m_{\R}) \left(0,\frac{1}{ad},\frac{1}{bd}\right). \\
  \end{aligned} \\
\]
By construction, the first point is \(7/6\) from \(\Lambda\).
Thus, by the triangle inequality,
\[\begin{aligned} \left| k \v - \Lambda \right|  & \ge \frac{7}{6}  - 2p |m - m_{\R}|
\left(\frac{1}{ad} + \frac{1}{bd} \right) \\
& \ge 1 + \frac{1}{6}  - p J(a b d) \left(\frac{1}{ad} + \frac{1}{bd} \right)  \end{aligned} \]
This is greater than one as long as
\[\frac{1}{6}   \ge p J(a b d) \left(\frac{1}{ad} + \frac{1}{bd} \right),\]
or rearranging a little bit more, as long as
\begin{equation}
\label{backtoback}
J(a b d) \le \frac{a  b d}{6p(a + b)}.
\end{equation}
Note that with \(15 p \ge b,a\), and \(33 \ge p\), this would certainly follow
if we knew that
\[J(a b d) \le \frac{a  b d}{198(495+495)} = \frac{a b d}{196020}.\]
But we know by Lemma~\ref{sqbound} that~\(J(a b d) \le 2 \sqrt{a b d}\).
Hence we are done as long as
\[2 \sqrt{a b d} \le   \frac{a b d}{196020},\]
which holds as soon as
\[a b d \ge 4 \cdot 196020^2 = 153695361600.\]
Hence we can assume that \(a b d\) is less than this quantity.
But then we know that \(a b d\) has at most \(10\) prime factors, since
\[P_{11} = 200560490130  
\ge 153695361600.\]
This allows us to use better bounds on \(J(a b d)\),
namely that \(J(a b d) \le 46\) by Lemma~\ref{better}, and thus \eqref{backtoback}
is satisfied (and we are done) as long as
\[a b d \ge 46 \cdot 6p(a+b),\]
which certainly is satisfied if
\[ d \ge 276 \cdot  \frac{p(a+b)}{ab}.\]
With \(495 \ge 15p \ge b,a \ge 1\), the RHS is maximized
with \(b=a=1\), and so we are done if
\(d \ge 18216\).
Hence we compute over all triples \(p,a,b,d\) with:
\begin{enumerate}
\item \(p = 2, \ldots, 33\),
\item \(a \le b \le 15p\) with \((a,b) = 1\) and \(a b \le 265p\).
\item \(\displaystyle{d \le \frac{276 p(a+b)}{a b}}\).
\end{enumerate}
Checking all these cases, we complete the proof of the Lemma.
\end{proof}

\subsection{Self-improvement}

We can now feed the results of this section back into Lemma~\ref{mbound4}
to prove a stronger result:

\begin{lemma}  \label{prep} If \((p,q,r)\) is any triple which does not satisfy the conclusion
of Theorem~\ref{maintheorem}, then we can assume:
\begin{enumerate}
\item \(\min(p,q,r) > 33\).
\item For any ordering, of the triple, if we let \(m = [p,q,r]/[p,q]\),
then we may assume that \(m \in \{1,2,3,5,6\}\).
\end{enumerate}
\end{lemma}

\begin{proof}
The first claim follows from Theorem~\ref{medium}.
The second claim follows from Lemma~\ref{mbound4}.
\end{proof}

\subsection{Increasing the range of \texorpdfstring{\(p\)}{p}}

Now that we know that \(p,q,r> 33\), we can improve
upon Lemma~\ref{medium} by
using Lemma~\ref{prep}.

\begin{lemma} \label{mediumplus} Suppose that \(p \le 885\).
Then Theorem~\ref{maintheorem}
holds for \((p,q,r) = (p,an,bn)\).
\end{lemma}

\begin{proof}
We proceed exactly as in Lemma~\ref{medium} except
now with the additional assumption that \(p > 33\).
Let \(B = 885\) be our bound for \(p\).
Since \(p > 33\),
we know by Lemma~\ref{prep}
that \(\max(a,b) \le 6p\), and following Lemma~\ref{bound}
we also find that:
\[\begin{aligned}
b & \le 6p, \\
a & \le 6p, \\
ab & \le 30p.
\end{aligned}
\]

Returning to~\eqref{backtoback}, we have
\begin{equation}
J(a b d) \le \frac{a  b d}{6p(a + b)}.
\end{equation}
This would follow if we knew that
if we knew that
\[J(a b d) \le \frac{a  b d}{6B(6B + 6B)} = 
\frac{a b c}{72 B^2}.\]
But we know by Lemma~\ref{sqbound} that~\(J(a b d) \le 2 \sqrt{a b d}\).
Hence we are done as long as
\[2 \sqrt{a b d} \le   \frac{a b d}{72 B^2},\]
which holds as soon as
\begin{equation} \label{withB}
a b d \ge 4 \cdot (72 B^2)^2 = 20736 B^4 = 
12720320883360000 \sim 1.2 \times 10^{16}.
 \end{equation}
 Now we have
 \[P_{14} = 13082761331670030 \sim 1.3 \times 10^{16} > 12720320883360000.\]
 This means that we may assume that \(abd\) as at most \(13\) factors.
 This allows us to use better bounds on \(J(a b d)\),
namely that \(J(a b d) \le 74\) by Lemma~\ref{better}, and thus \eqref{backtoback}
is satisfied (and we are done) as long as
\[a b d \ge 74 \cdot 6B \cdot (12B) \ge 74 \cdot 6p(a+b),\]
and so we are done if
\[a b d \ge 5328 B^2 = 4173022800 < 6469693230 = P_{10},\]
from which we deduce that \(a b d\) has at most \(9\) prime factors.
Repeating this process, we
 feed this bound back into the the same argument to deduce that
 \(J(a b d) \le 40\). Thus we are done as long as
 \[a b d \ge 240p(a+b),\]
which certainly is satisfied if
\[ d \ge 240 \cdot  \frac{p(a+b)}{ab}.\]
Thus it remains to loop over all \(p,a,b,d\) with:
\begin{enumerate}
\item \(p = 34, \ldots, B\),
\item \(a \le b \le 6p\) with \((a,b) = 1\) and \(\displaystyle{\frac{a b}{(ab,p)} \le 30}\).
\item \(\displaystyle{d \le \frac{240 p(a+b)}{a b}}\).
\end{enumerate}
We check that this does not lead to any new triples.
For remarks about the computational aspect of this proof,
see~\S~\ref{comment}.
\end{proof}

\section{Completing the Argument}
\label{sec:resolution}
We now combine the results of the last two sections to complete the argument.
We begin with a special case:

\begin{lemma}  \label{odd}
If \(p,q,r\) are all odd, then Theorem~\ref{maintheorem}
applies for this triple.
\end{lemma}

\begin{proof}
Note that \(pqr\) is odd so we can let \(k\) be one of
\(\frac{pqr + 1}{2}, \frac{pqr-1}{2},\)
both of which are prime to \(pqr\) and one of which is odd, so prime to \(n = [2,p,q,r]\). For that choice, we have
\[ \frac{k}{p} = \frac{q r}{2}  \pm \frac{1}{2p} \equiv \frac{1}{2} \pm \frac{1}{2p} \bmod \Z.\]
Thus we find that
\[ \left| \left( \frac{k}{p},\frac{k}{q},\frac{k}{r} \right) - \Lambda \right|
\ge \frac{3}{2} - \frac{1}{2p} - \frac{1}{2q} - \frac{1}{2r}
\ge \frac{3}{2} - \frac{3}{6} \ge 1,\]
as soon as \(\min(p,q,r)  \ge 6\), which we can assume by Lemma~\ref{medium}.
\end{proof}

\subsection{Proof of Theorem~\ref{maintheorem}}

Assume that \((p,q,r)\) is a counter example to Theorem~\ref{maintheorem},
and let \(n = [2,p,q,r]\).
By Lemma~\ref{mediumplus} we know that \(\min(p,q,r) = B > 885\).
By Lemma~\ref{odd}, we may also assume that \(n = [2,p,q,r] =[p,q,r]\).
By Lemma~\ref{circular}, with \(n = [p,q,r]\),
we therefore deduce that
\[n \ge 117288381359406970983270 = 1.1 \times 10^{23},\]
and moreover that \(n\) has at least \(r \ge 18\) prime factors.
As in the proof of Lemma~\ref{mediumplus},
we write \((p,q,r)=(p,ad,bd)\) with \(p = B\) and
\[\begin{aligned}
b & \le 6B, \\
a & \le 6B, \\
ab & \le 30B,
\end{aligned}
\]
and we are done providing that
\begin{equation}
J(a b d) \le \frac{a  b d}{6p(a + b)},
\end{equation}
and hence done if
\begin{equation}
J(a b d) \le \frac{a  b d}{72 B^2}.
\end{equation}

We know that \(m = [p,ad,bd]/[ad,bd] \le 6\), and so \(n\)
divides \(6abd\). Also the primes dividing \(n\) are the
primes dividing \(abd\) so \(J(n) = J(abd)\). 
Hence we are done if
\[J(n) \le \frac{n}{432 B^2}.\]
On the other hand, by Lemma~\ref{circular}, we are also done if
\[J(n) \le \frac{2B}{15}.\]
if neither of these are satisfied, then
\[J(n)^3 = J(n) \cdot J(n)^2 > \frac{n}{432 B^2} \cdot \left(\frac{2B}{15}\right)^2
= \frac{n}{24300}\]
But this is a contradiction by Lemma~\ref{sqbound}
as soon as \(r \ge 16\), and as noted above we may assume that \(r \ge 18\).
This completes the proof of Theorem~\ref{maintheorem}. \qed

\subsection{Acknowledgments}

The genesis of this paper  was
a project  of the second author
supervised by the first author. The initial problem
was to consider triples of the form \((p,q,r) = (6,d,d)\),
which can be handled in a completely elementary way.
The case when \(p=6\) can be thought of as the ``worst case''
in light of  Lemma~\ref{onedim}; it already includes \(8\)
of the \(11\) elements of the Hilbert Series.
It is a short step from this  to the case of \((p,q,r) = (6,ad,bd)\)
for small \((a,b)=1\), and then to Lemma~\ref{medium}, and finally to
Lemma~\ref{mediumplus}. 
At the same time, one can also give a completely elementary
proof of Theorem~\ref{5lemma}  (or rather its analogue) for hyperplanes rather than lines,
and additionally to prove Theorem~\ref{5lemma} for all but finitely many
(explicit) lines on any given hyperplane, using
a version of Lemma~\ref{elementary}. This is not quite good enough
to reduce Theorem~\ref{5lemma} to a finite calculation, since one
has to worry about possible exceptions on every hyperplane ---
  boundary
cases like the line \((t/2,t/3,t/6)\)  do lie on infinitely many hyperplanes. 
The argument
required to reduce Theorem~\ref{5lemma} to finitely many hyperplanes (using Fourier analysis or the geometry of numbers)
 is the only non-elementary step in this paper.
We thank  Andrew Sutherland at MIT for providing us access to
his \(128\) core machine with a \texttt{magma} license,
and with help in setting up the scripts to run in parallel.

\appendix

\section{Remarks on computations}
In this section, we give some more details on our explicit
computations, with links to files on \texttt{github}.
We also make some remarks on computational efficiently.
Suppose we have a triple of integers \((p,q,r)\) and we wish to
find either:
\begin{enumerate}
\item An integer \((k,n)=1\) with \(n=[2,p,q,r]\) such that
\[ \left| \left(\frac{k}{p},\frac{k}{q},\frac{k}{r}\right) - \Lambda \right| \ge 1\]
\item An real number \(t\) such that
\[ \left| \left(\frac{t}{p},\frac{t}{q},\frac{t}{r}\right) - \Lambda \right| \ge 1 + \frac{1}{5}.\]
\end{enumerate}
Following~\cite{Curt}, a natural approach to the first problem
is simply to repeatedly pick a random integer \((k,n)=1\) and check
if the condition is satisfied. 
We expect
(from~\cite{Bilu})
from the spectral gap condition that this is satisfied for a large positive percentage
of \(k\) (independent of \(n\)), unless \((p,q,r)\) is a triple for
which no such \(k\) exists. Hence this provides
an excellent probabilistic test ---  if we run  this up to \(1000\) times
until it finishes, this is exceedingly likely to find such a \(k\) if \((p,q,r)\)
is not in the Hilbert Series (or some other unknown exception 
before Theorem~\ref{maintheorem} is proved). Moreover, there is no issue
with false positives, only false negatives --- if a triple does pass
all \(1000\) tests, we can check it theoretically by hand (this never happened). 

The approach in the second case is exceedingly similar. Instead of choosing \(t\)
to be a random integer prime to \(n\), we choose a random integer in \([-1000n,1000n]\)
and then divide by \(1000)\) to get a rational number in \([-n,n]\). This works equally well in practice.
Computing \(J(n)\) for relatively small values of \(n\) (say \(n \le 10000\))
is very easy to do directly and so we make no further comments
on these calculations. The most difficult computations 
we use are the upper bounds for \(J(n)\) for integers \(n\) with
\(r\) prime divisors --- but these computations have
already been done in other papers
which we may simply cite.)

\subsection{The case of fixed primes \texorpdfstring{\(p\)}{p}} \label{comment}
By far, the longest computation was the verification
that \(\min(p,q,r) \ge 885\) in the proof of Lemma~\ref{mediumplus}.
Our original program in \texttt{magma}~\cite{MR1484478}
first computed in the range to \(B \le 33\), and then \(B \le 100\),
the second computation finishing in  slightly under one day.
The program for \(B = 101 \ldots 885\) was set running  with
a single core.
On the other hand, the computations for every value of \(p\) are completely
independent, and so in particular this is a very parallelizable computation.
One issue is that \texttt{magma} licenses  do not easily transfer
to cloud machines, so one would have
to port the code to \texttt{c++}.
 Andrew Sutherland, however,
  generously provided use
of \(100\) cores on his \(128\) core machine which did
have  a \texttt{magma} license. We then divided
up the computation for \(B=34, \ldots 885\) into \(852\) individual computations
and set them running late overnight  in parallel on \(100\) cores. By  the
next morning the computations had  long since
been  completed. The original   computation (using one core)
was left to puff away even after the parallel computation had long since finished. 
It eventually completed its task (finding no further triples) after \(267\) hours.
 
 \subsection{The co-dimension two lines of \S~\ref{sec:codim2}}
 
For some computations in~\S~\ref{sec:fourier} we used mathematica. Computing
the values of \(E[(e(t)-24)^{12}]\) for the exceptional
hyperplanes takes a little less than \(30\) minutes. 
Forming the cross products of
\[\binom{6337}{2} = 20075616\]
pairs of elements, and then removing duplicates after scaling,
taking absolute values, and dividing by the GCD. Operating
on lists of this size in \texttt{mathematica} took well under \(10\) minutes.
Testing the exceptional cases can all be done in the order of minutes.
The explicit computer scrips together with
input and output files can be found here~\cite{git}.

\bibliographystyle{amsalpha}
\bibliography{Fricke}

 \end{document}